\documentclass[a4paper,10pt]{amsart}
\usepackage[utf8]{inputenc}
\usepackage[T1]{fontenc}

\usepackage{todonotes}
\usepackage{enumitem}

\usepackage{amsmath, amsthm, amsfonts, amssymb, calrsfs}
\usepackage{caption, subfig}
\usepackage{array, multirow, multicol}
\usepackage{makecell}
\setcellgapes{1pt}
\makegapedcells

\usepackage{tikz-cd, tikz}
\usepackage{graphicx}
\usepackage{yfonts}
\usepackage{xcolor}

\usepackage[hyperindex=true, colorlinks=true, linkcolor=blue, filecolor=black,
citecolor=blue, urlcolor=blue, bookmarks=true, pagebackref=true,
pdftitle={Stability of equivariant logarithmic tangent sheaves}]{hyperref}

\theoremstyle{plain}{
    \newtheorem{theorem}{Theorem}[section]
    \newtheorem{lem}[theorem]{Lemma}
    \newtheorem{cor}[theorem]{Corollary}
    \newtheorem{prop}[theorem]{Proposition}
    
}

\theoremstyle{definition}{
    \newtheorem{defn}[theorem]{Definition}
    \newtheorem{defn-thm}[theorem]{Definition-Theorem}
    \newtheorem{example}[theorem]{Example}
    \newtheorem{notation}[theorem]{Notation}
}

\theoremstyle{remark}{
    \newtheorem{rem}[theorem]{Remark}

}
\def\lin{\mathrm{lin}}
\def\reg{\mathrm{reg}}
\newcommand{\C}{\mathbb{C}}
\newcommand{\E}{\mathbb{E}}
\newcommand{\F}{\mathbb{F}}
\newcommand{\K}{\mathbb{K}}
\newcommand{\bL}{\mathbb{L}}
\newcommand{\N}{\mathbb{N}}
\newcommand{\bP}{\mathbb{P}}
\newcommand{\Q}{\mathbb{Q}}
\newcommand{\R}{\mathbb{R}}
\newcommand{\Z}{\mathbb{Z}}

\newcommand{\cE}{\mathcal{E}}
\newcommand{\cF}{\mathcal{F}}
\newcommand{\cG}{\mathcal{G}}
\newcommand{\cO}{\mathcal{O}}
\newcommand{\cT}{\mathcal{T}}
\newcommand{\cV}{\mathcal{V}}
\newcommand{\fK}{\mathfrak{K}}
\newcommand{\rP}{\mathrm{P}}
\newcommand{\rQ}{\mathrm{Q}}
\newcommand{\rV}{\mathrm{V}}
\newcommand{\rW}{\mathrm{W}}
\def\Id{\mathrm{Id}}
\def\<{\langle}
\def\>{\rangle}
\newcommand{\nothing}{\varnothing}
\newcommand{\imp}{\longrightarrow}
\DeclareMathOperator{\card}{card}
\DeclareMathOperator{\vol}{vol}
\DeclareMathOperator{\codim}{codim}
\DeclareMathOperator{\rk}{rk}

\DeclareMathOperator{\Hom}{Hom}
\DeclareMathOperator{\Span}{Span}
\DeclareMathOperator{\Spec}{Spec}
\DeclareMathOperator{\WDiv}{WDiv}
\DeclareMathOperator{\Cl}{Cl}
\DeclareMathOperator{\Pic}{Pic}
\DeclareMathOperator{\Amp}{Amp}
\DeclareMathOperator{\Stab}{Stab}
\DeclareMathOperator{\sStab}{sStab}
\DeclareMathOperator{\Cone}{Cone}
\DeclareMathOperator{\Conv}{Conv}
\newcolumntype{R}[1]{>{\raggedleft\arraybackslash} p{#1}}
\newcolumntype{L}[1]{>{\raggedright\arraybackslash} p{#1}}
\newcolumntype{C}[1]{>{\centering\arraybackslash} p{#1}}
\begin{document}

\title[Stability of equivariant logarithmic tangent sheaves]{Stability of
equivariant logarithmic tangent sheaves on toric varieties of Picard rank two}

\author[A. Napame]{Achim NAPAME}

\address{Univ Brest, UMR CNRS 6205, Laboratoire de Mathématiques de Bretagne
Atlantique, France}

\email{achim.napame@univ-brest.fr}

\keywords{Toric varieties, logarithmic tangent sheaves, slope-stability}
%\subjclass{14M25}
%\date{\today}

\begin{abstract}
For an equivariant log pair $(X, D)$ where $X$ is a normal toric variety and $D$ a
reduced Weil divisor, we study slope-stability of the logarithmic tangent sheaf
$\mathcal{T}_{X}(- \log D)$. We give a complete description of divisors $D$ and
polarizations $L$ such that $\mathcal{T}_{X}(- \log D)$ is (semi)stable with respect
to $L$ when $X$ has a Picard rank one or two.
\end{abstract}

\maketitle

\section{Introduction}\label{sec:intro}
The notion of slope-stability was first introduced by Mumford \cite{Mum62} in his
construction of moduli spaces of vector bundles over a curve. This notion was generalized
in higher dimension by Takemoto \cite{Tak72}.
A vector bundle, or more generally a torsion-free sheaf $\cE$ on a complex projective
variety $X$ is said to be slope-stable (resp. semistable) with respect to a polarization
$L$, if for any proper coherent subsheaf $\cF$ of $\cE$ with $0 < \rk(\cF) < \rk(\cE)$,
one has $\mu_{L}(\cF) < \mu_{L}(\cE)$ (resp. $\mu_{L}(\cF) \leq \mu_{L}(\cE)$)
where the {\it slope} of $\cE$ with respect to $L$ is given by
$$
\mu_{L}(\cE) = \dfrac{c_1(\cE) \cdot L^{\dim(X)-1}}{\rk(\cE)} .
$$
As the study of stability of reflexive sheaves is a difficult problem, we are interested
by the category of torus equivariant reflexive sheaves over normal toric varieties.
Using the description of equivariant reflexive sheaves over toric varieties in terms
of families of filtrations given by Klyachko \cite{Kly90} and Perling \cite{Per04},
Kool in \cite[Proposition 4.13]{Koo11} showed that it is enough to compare slopes for
equivariant and reflexive saturated subsheaves.

Tangent sheaves are natural examples of equivariant reflexive sheaves on normal toric
varieties. Using its equivariant structure, Hering--Nill--S\"uss \cite{HNS19} and
Dasgupta--Dey--Khan \cite{DDK20} studied slope-stability of the tangent bundle of smooth
projective toric varieties of Picard rank one or two.
Inspired by Iitaka's philosophy, in this paper, we extend the results of
\cite{DDK20, HNS19} to the case of equivariant logarithmic pairs $(X, D)$.
More precisely, if $X$ is a normal toric variety and $D$ a reduced snc (simple normal
crossing) divisor such that the logarithmic tangent sheaf $\cT_{X}(- \log D)$ is
equivariant, we are interested by the set of polarizations $L$ on $X$ such that
$\cT_{X}(- \log D)$ is (semi)stable with respect to $L$.

We first note that the logarithmic tangent sheaf $\cT_{X}(- \log D)$ is equivariant if
and only if $D$ is a torus invariant divisor of $X$.
For a toric variety $X$ with fan $\Sigma$ in $N \otimes_\Z \R$, we denote by $D_\rho$
the torus invariant divisor of $X$ corresponding to the ray $\rho \in \Sigma(1)$
(see Section \ref{sec:background} for precise definitions).
Then we have:

\begin{theorem}\label{theo:filt-logtangent-intro}
Let $\Delta \subseteq \Sigma(1)$ and $D= \sum_{\rho \in \Delta} D_\rho$.
The family of filtrations
$\left(E, \{ E^{\rho}(j) \}_{\rho \in \Sigma(1), \, j \in \Z} \right)$
of the logarithmic tangent sheaf $\cT_{X}(-\log D)$ is given by
$$
E^{\rho}(j) = \left\lbrace
\begin{array}{ll}
0 & \text{if}~ j \leq -1 \\ N \otimes_{\Z} \C & \text{if}~ j \geq 0
\end{array}
\right.
\qquad \text{if}~ \rho \in \Delta
$$
and by
$$
E^{\rho}(j) = \left\lbrace
\begin{array}{ll}
0 & \text{if}~ j \leq -2 \\
\Span(u_{\rho}) & \text{if}~ j = -1 \\
N \otimes_{\Z} \C & \text{if}~ j \geq 0
\end{array}
\right.
\qquad \text{if}~ \rho \notin \Delta
$$
where $u_\rho \in N$ is the minimal generator of the ray $\rho$.
\end{theorem}

\begin{rem}
We will see in Section \ref{sec:filtration-logtangent} that if $\Delta = \Sigma(1)$,
then $\cT_{X}(- \log D)$ is isomorphic to the trivial sheaf of rank $\dim(X)$ and if
$\Delta = \nothing$, then $\cT_{X}(- \log D)$ is the tangent sheaf $\cT_{X}$.
\end{rem}

By Theorem \ref{theo:filt-logtangent-intro} and the fact that
$|\Sigma(1)| = \dim(X) + \rk( \Cl(X))$
on complete normal toric varieties $X$, we show that:

\begin{prop}\label{prop:logtangent-unstability-intro}
If $1 + \rk(\Cl(X)) \leq |\Delta| \leq |\Sigma(1)|-1$, then for any polarization $L$, the
logarithmic tangent sheaf $\cT_{X}(- \log D)$ is unstable with respect to $L$.
\end{prop}

According to this proposition, it is therefore sufficient to study the stability of
$\cT_{X}(- \log D)$ when $|\Delta| \leq \rk(\Cl(X))$. Thus, in this paper we study the
case where $X$ is smooth, $\rk \Cl(X) \in \{1, 2 \}$ and $1 \leq |\Delta| \leq \rk \Cl(X)$.
Note that the only smooth projective toric variety with Picard number one is the
projective space $\bP^n$.

\begin{prop}
Let $D$ be an invariant hyperplane section of $\bP^n$. Then, the logarithmic tangent
sheaf $\cT_{\bP^n}(- \log D)$ is polystable with respect to $\cO_{\bP^n}(1)$.
\end{prop}

By \cite[Theorem 1]{Kle88}, every smooth toric variety of Picard rank two is of the form
$X = \bP ( \cO_{\bP^s} \oplus \bigoplus_{i=1}^{r} \cO_{\bP^s}(a_i) )$ with
$r, s \in \N^{\ast}$ and $a_1, \ldots, a_r \in \N$ such that $a_1 \leq \ldots \leq a_r$.
Moreover, $X$ blown down to $\bP^{r+s}$ if and only if
$(a_1, \ldots, a_r) = (0, \ldots, 0, 1)$.
We denote by $\pi: X \rightarrow \bP^s$ the projection map.
Let $\cV$ be a vector bundle associated to the locally free sheaf
$$
\cO_{\bP^s} \oplus \cO_{\bP^s}(-a_1) \oplus \ldots \oplus \cO_{\bP^s}(-a_r) .
$$
Then the irreducible invariant divisors of $X$ are given by
$$
\left\lbrace
\begin{array}{ll}
D_{w_j} = \pi^{-1}( \{(z_0: \ldots : z_s) \in \bP^s : z_j = 0 \})
& \text{for}~ 0 \leq j \leq s \\
D_{v_i} = \{s_i = 0\} & \text{for}~ 0 \leq i \leq r
\end{array}
\right.
$$
where the $\{s_i = 0\}$ are the relative hyperplane sections associated to the line
subbundles of $\cV^\vee$.
If $L = \pi^\ast \cO_{\bP^s}(\alpha) \otimes \cO_{X}(\beta)$ is a polarization of $X$,
according to the value of $\nu := \alpha/\beta$, in Tables
\ref{tab:stab-logtgt-prodprojective}, \ref{tab:stab-logtgt-toricpic2-1},
\ref{tab:stab-logtgt-toricpic2-2} and \ref{tab:stab-logtgt-toricpic2-3}, we give a
complete classification of reduced divisors $D$ and polarizations $L$ on $X$ such that the
equivariant logarithmic tangent sheaf $\cT_{X}(- \log D)$ is (semi)stable with respect to
$L$.
In particular:

\begin{prop}
Let $X = \bP ( \cO_{\bP^s} \oplus \bigoplus_{i=1}^{r} \cO_{\bP^s}(a_i) )$ with
$a_1 = \ldots = a_r = 0$. Then for any
\begin{align*}
D \in \{D_{v_i} : 0 \leq i \leq r \} & \cup \{D_{w_j} : 0 \leq j \leq s \} \\
& \cup \{D_{v_i} + D_{w_j} : 0 \leq i \leq r, 0 \leq j \leq s \}
\end{align*}
$\cT_{X}(- \log D)$ is polystable with respect to $L$ if and only if $L$ is a power of
the polarization corresponding to $- (K_X + D)$.
\end{prop}

For $a_r \geq 1$, we show that:

\begin{theorem}\label{theo:logtgt-intro-Dvr}
Let $X= \bP( \cO_{\bP^s} \oplus \bigoplus_{i=1}^{r} \cO_{\bP^s}(a_i) )$ with $a_r \geq 1$.
There are $\alpha, \beta \in \N^\ast$ such that the logarithmic tangent sheaf
$\cT_{X}(- \log D_{v_r})$ is (semi)stable with respect to
$\pi^{\ast} \cO_{\bP^s}(\alpha) \otimes \cO_{X}(\beta)$ if and only if $a_r = 1$ and
$a_{r-1} = 0$. Moreover, if $a_r = 1$ and $a_{r-1} = 0$, then
$\cT_{X}(- \log D_{v_r})$ is stable (resp. semistable) with respect to
$\pi^{\ast} \cO_{\bP^s}(\alpha) \otimes \cO_{X}(\beta)$ if and only if
$0 < \frac{\alpha}{\beta} < \nu_0$ (resp. $0 < \frac{\alpha}{\beta} \leq \nu_0$)
where $\nu_0$ is the unique positive root of
$$
\rP_0(x) = \sum_{k = 0}^{s-1} \dbinom{s+r-1}{k} x^k - s \dbinom{s+r-1}{s} x^s ~.
$$
\end{theorem}

Theorem \ref{theo:logtgt-intro-Dvr} can be seen as an extension
of \cite[Theorem 1.4]{HNS19} to the case of the logarithmic pair $(X, D_{v_r})$.
Indeed, in \cite[Theorem 1.4]{HNS19} it is shown that for
$X= \bP( \cO_{\bP^s} \oplus \bigoplus_{i=1}^{r} \cO_{\bP^s}(a_i) )$ with $a_r \geq 1$,
there are $\alpha, \beta \in \N^\ast$ such that the tangent sheaf $\cT_{X}$ is
(semi)stable with respect to $\pi^{\ast} \cO_{\bP^s}(\alpha) \otimes \cO_{X}(\beta)$
if and only if $a_r = 1$ and $a_{r-1} = 0$.

\begin{rem}
In the logarithmic case, it is not just varieties $X$ which blown down to $\bP^{r+s}$
which admit divisors $D$ such that $\cT_{X}(- \log D)$ is stable.
If $X = \bP( \cO_{\bP^s} \oplus \bigoplus_{i=1}^{r} \cO_{\bP^s}(a_i))$
with $a_r \geq 1$ and $(a_1, \ldots, a_r)$ not necessarily equal to $(0, \ldots, 0, 1)$,
in Theorems \ref{theo:logtgt-toricpic-Dv0} and \ref{theo:logtgt-toricpic-Dv01},
we will show that there are polarizations $L$ on $X$ such that $\cT_{X}(- \log D_{v_0})$
is (semi)stable with respect to $L$ if and only if $a_1 = \ldots = a_r$ and
$(r-1)a_{r} < (s+1)$.
\end{rem}

\begin{rem}
For these studies of stability when $\rk \Cl(X) = 2$, we use the calculations made in
\cite{HNS19} but we simplify their arguments using Lemma
\ref{lem:stability-logtangent-toricpic2}.
\end{rem}

\subsection*{Organization}
In Section \ref{sec:background} we recall the necessary background on toric varieties,
equivariant reflexive sheaves and their families of filtrations. We also recall the
notions of slope-stability. In Section \ref{sec:log-sheaves}, we study the logarithmic
tangent sheaf. We prove Theorem \ref{theo:filt-logtangent-intro} and Proposition
\ref{prop:logtangent-unstability-intro}.
Sections \ref{sec:stability-logtangent} and \ref{sec:stab-logtangent-toricpic2} deal
with the study of the stability of $\cT_{X}(- \log D)$ when $\rk \Cl(X) = 2$.
In Section \ref{sec:logdpezzo-pairs}, we apply the results of the paper on Hirzebruch
surfaces.

\subsection*{Acknowledgments}
I would like to thank my advisor Carl \textsc{Tipler} for our discussions on this
subject and also Henri \textsc{Guenancia} for some references.

\section{Toric varieties, equivariant sheaves and stability notions}
\label{sec:background}

In this section, we present the different notions that will be discussed in this paper:
toric varieties \cite{CLS}, equivariant sheaves \cite{Per04} and stability of sheaves
\cite{Tak72}.

\subsection{Normal toric varieties}\label{sec:toric-varieties}
A $n$-dimensional {\it toric variety} is an irreducible variety $X$ containing a torus
$T \simeq (\C^{\ast})^n$ as a Zariski open subset such that the action of $T$ on itself
extends to an algebraic action of $T$ on $X$.

Let $N$ be a rank $n$ lattice and $M = \Hom_{\Z}(N, \Z)$ be its dual with pairing
$\< \cdot, \cdot \> : M \times N \rightarrow \Z$. Then $N$ is the {\it lattice of
one-parameter subgroups} of the $n$-dimensional complex torus
$T_N:= N \otimes_{\Z} \C^\ast = \Hom_{\Z}(M, \C^\ast)$.
We call $M$ the {\it lattice of characters} of $T_N$.
For $\K=\text{$\R$ or $\C$}$, we define $N_\K = N \otimes_\Z \K$ and
$M_\K = M \otimes_\Z \K$. We denote by $\chi^m: T_N \rightarrow \C^\ast$ the character
corresponding to $m \in M$ and by $\lambda^u: \C^\ast \rightarrow T_N$ the one-parameter
subgroup corresponding to $u \in N$.

A {\it fan} $\Sigma$ in $N_{\R}$ is a set of rational strongly convex polyhedral cones
in $N_{\R}$ such that:
\begin{itemize}
\item Each face of a cone in $\Sigma$ is also a cone in $\Sigma$;
\item The intersection of two cones in $\Sigma$ is a face of each.
\end{itemize}
We will denote $\tau \preceq \sigma$ the inclusion of a face $\tau$ in $\sigma \in \Sigma$.
A cone $\sigma$ in $N_{\R}$ is {\it smooth} if its minimal generators form part of a
$\Z$-basis of $N$. We say that $\sigma$ is {\it simplicial} if its minimal generators are
linearly independent over $\R$.
A fan $\Sigma$ is {\it smooth} (resp. {\it simplicial}) if every cone $\sigma$ in $\Sigma$
is smooth (resp. {\it simplicial}).
The {\it support} of $\Sigma$ is given by
$|\Sigma| := \bigcup_{\sigma \in \Sigma}{\sigma}$
and we say that $\Sigma$ is {\it complete} if $|\Sigma| = N_{\R}$.

\begin{notation}
For a finite subset $S \subseteq N$, we denote by $\Cone(S)$ the cone generated by $S$.
For a fan $\Sigma$, we denote by
\begin{itemize}
\item $\Sigma(r)$ the set of $r$-dimensional cones of $\Sigma$;
\item $u_{\rho} \in N$ the minimal generator of $\rho \in \Sigma(1)$.
\end{itemize}
Elements of $\Sigma(1)$ will be called {\it rays}.
\end{notation}

For $\sigma \in \Sigma$, let $U_{\sigma} = \Spec( \C[S_{\sigma}])$ where
$\C[S_{\sigma}]$ is the semi-group algebra of
$$
S_{\sigma} = \sigma^{\vee} \cap M = \{ m \in M \, : \, \langle m, \, u \rangle \geq 0
~ \text{for all}~ u \in \sigma \} ~~.
$$
If $\sigma, \, \sigma' \in \Sigma$, we have $U_{\sigma} \cap U_{\sigma'} =
U_{\sigma \cap \sigma'}$.
We denote by $X_{\Sigma}$ the toric variety associated to a fan $\Sigma$;
$X_{\Sigma}$ is obtained by gluing the affine charts $(U_{\sigma})_{\sigma \in \Sigma}$.
The variety $X_{\Sigma}$ is normal and its torus is $T_N$.
As every separated normal toric variety comes from a fan, from now on, a normal toric
variety will be defined by a fan.

Let $X$ be the toric variety associated to a fan $\Sigma$ in $N_{\R}$.
For any $\sigma \in \Sigma$, there is a point $\gamma_\sigma \in U_{\sigma}$ called the
{\it distinguished point} of $\sigma$ such that the torus orbit $O(\sigma)$ corresponding
to $\sigma$ is given by $O(\sigma) = T \cdot \gamma_\sigma$. We will use the following
result:

\begin{theorem}[Orbit-Cone Correspondence {\cite[Theorem 3.2.6]{CLS}}]
\label{theo:orbit-cone}
Let $X$ be the toric variety associated to a fan $\Sigma$ with torus $T$. Then
\begin{enumerate}
\item
There is a bijective correspondence
$$
\begin{array}{rcl}
\{ \text{Cones $\sigma$ in $\Sigma$} \} & \longleftrightarrow &
\{\text{$T$-orbits in $X$} \} \\
\sigma & \longleftrightarrow & O(\sigma)
\end{array}
$$
with
$\dim O(\sigma) = \dim N_{\R} - \dim \sigma$.
\item
The affine open subset $U_{\sigma}$ is the union of orbits
$$U_{\sigma} = \bigcup_{\tau \preceq \sigma}{O(\tau)}.$$
\item
$\tau \preceq \sigma$ if and only if $O(\sigma) \subseteq \overline{O(\tau)}$, and
$$\overline{O(\tau)} = \bigcup_{\tau \preceq \sigma}{O(\sigma)}$$ where
$\overline{O(\tau)}$ denotes the closure in both the classical and Zariski topologies.
\end{enumerate}
\end{theorem}

\begin{notation}
For any $\rho \in \Sigma(1)$, we set $D_\rho = \overline{O(\rho)}$.
\end{notation}

For any $\rho \in \Sigma(1)$, $D_\rho$ defines a $T$-invariant Weil divisor of $X$.
Divisors of the form $\sum_{\rho \in \Sigma(1)} a_{\rho} D_{\rho}$ are precisely the
invariant divisors under the torus action on $X$. Thus,
$$
\WDiv_{T}(X) := \bigoplus_{\rho \in \Sigma(1)} \Z D_\rho
$$
is the group of invariant Weil divisors on $X$. In particular,

\begin{theorem}[{\cite[Theorem 8.2.3]{CLS}}]\label{theo:canonical-divisor}
The canonical divisor of a toric variety $X_\Sigma$ is the torus invariant Weil divisor
$$
K_{X_\Sigma} = - \sum_{\rho \in \Sigma(1)} D_\rho .
$$
\end{theorem}

By \cite[Proposition 4.1.2]{CLS}, for any $m \in M$, the character $\chi^m$ is a rational
function on $X_\Sigma$, and its divisor is given by
\begin{equation}\label{eq:divisor-of-character}
\mathrm{div}(\chi^m) =\sum_{\rho \in \Sigma(1)}{\< m, u_{\rho} \>} D_{\rho},
\end{equation}
so $\mathrm{div}(\chi^m)$ defines an invariant principal divisor of $X_\Sigma$.
A normal toric variety $X_\Sigma$ has a {\it torus factor} if and only if the set
$\{u_{\rho} : \rho \in \Sigma(1) \}$ do not span $N_{\R}$.
If $X_\Sigma$ has no torus factor, then by \cite[Theorem 4.1.3]{CLS} we have
the exact sequence
\begin{equation}
0 \imp M \imp \WDiv_{T}(X_\Sigma) \imp \Cl(X_\Sigma) \imp 0
\end{equation}
where the map $M \rightarrow \WDiv_{T}(X_\Sigma)$ is given by Equation
(\ref{eq:divisor-of-character}). Therefore,

\begin{cor}\label{cor:number-rays}
If $X_\Sigma$ has no torus factor, then $|\Sigma(1)| = \dim(X_\Sigma) + \rk \Cl(X)$.
\end{cor}

We recall that a {\it lattice polytope} $\Conv(S)$ in $M_\R$ is the convex hull of a
finite set $S \subseteq M$. A Cartier divisor
$D = \sum_{\rho \in \Sigma(1)}{a_\rho D_\rho}$ on a complete toric variety $X_\Sigma$
gives the lattice polytope
$$
P_D = \{ m \in M_\R : \<m, u_\rho \> \geq - a_\rho \} \subseteq M_\R ~.
$$
If $P$ is a full dimensional lattice polytope in $M_\R$ given by
\begin{equation}\label{eq:polytope-form}
P = \{m \in M_\R : \<m, u_F \> \geq - a_F ~\text{for all facets $F$ of $P$} \}
\end{equation}
where $u_F \in N$ is an {\it inward-pointing normal} of the facet $F$ and $a_F \in \Z$,
we define the fan $\Sigma_P$ of $P$ by
$$
\Sigma_P = \{ \Cone(u_F : \text{$F$ contains $Q$}): \text{$Q$ is a face of $P$} \} ~.
$$
For any facet $F$ of $P$, we denote by $D_F$ the invariant divisor of the toric variety
$X_{\Sigma_P}$ corresponding to the ray $\Cone(u_F)$ and we set
$$
D_P = \sum_{F \preceq P} a_F D_F ~.
$$
So we have:

\begin{theorem}[{\cite[Theorem 6.2.1]{CLS}}]
\label{theo:polytopes-toric-varieties}
Let $X$ be a toric variety given by a complete fan $\Sigma$. Then, the map
$$
\begin{array}{ccc}
\left\lbrace
\begin{tabular}{@{}c@{}}
Torus invariant ample \\ divisor on $X$
\end{tabular}
\right\rbrace
&
\longrightarrow
&
\left\lbrace
\begin{tabular}{@{}c@{}}
Full dimensional lattice polytope\\ $P$ in $M_\R$ such that $\Sigma_P = \Sigma$
\end{tabular}
\right\rbrace
\\
D & \longmapsto & P_D
\end{array}
$$
is a bijective correspondence.
\end{theorem}

Let $P$ be the polytope corresponding to an invariant ample divisor $D$ on $X_\Sigma$.
For each $\rho \in \Sigma(1)$ we denote by $P^{\rho}$ the facet of $P$ corresponding to
the ray $\rho \in \Sigma(1)$.
We recall that a lattice $\mathrm{M}$ defines a measure $\nu$ on $\mathrm{M}_{\R}$ as the
pullback of the Haar measure on $\mathrm{M}_{\R}/ \mathrm{M}$. It is determined by the
properties
\begin{enumerate}
\item[i.] $\nu$ is translation invariant,
\item[ii.] $\nu( \mathrm{M}_{\R}/ \mathrm{M} ) = 1$.
\end{enumerate}
For all $\rho \in \Sigma(1)$, we denote by $\vol(P^{\rho})$ the volume of $P^{\rho}$
with respect to the measure determined by the affine span of $P^{\rho} \cap M$.

\begin{prop}[{\cite[Section 11]{Dan78}}]
\label{prop:degree-invariant-divisor}
Let $(X_\Sigma, D)$ be a polarized toric variety corresponding to a lattice polytope $P$.
For all $\rho \in \Sigma(1)$, $\vol(P^{\rho}) = D_\rho \cdot D^{n-1}$.
\end{prop}

\subsection{Smooth toric varieties of Picard rank two}
\label{sec:toricpicard-two}
Let $X$ be a smooth toric variety of dimension $n$ with fan $\Sigma$ in $\R^n$ such that
$\rk \Pic(X) = 2$. By \cite[Theorem 7.3.7]{CLS} due to Kleinschmidt \cite{Kle88},
there are $r, s \in \N^{\ast}$ with $r + s = n$ and $a_1, \ldots, a_r \in \N$ with
$a_1 \leq a_2 \leq \ldots \leq a_r$ such that
\begin{equation}\label{eq:toricpicard-two}
X = \bP \left( \cO_{\bP^s} \oplus \bigoplus_{i=1}^{r} \cO_{\bP^s}(a_i) \right).
\end{equation}
We denote by $\pi : X \rightarrow \bP^s$ the projection to the base $\bP^s$.
By \cite[Section 7.3]{CLS}, the rays of $\Sigma$ are given by the half-lines generated by
$w_0, w_1, \ldots, w_s$, $v_0, v_1, \ldots, v_r$ where
$(w_1, \ldots, w_s)$ is the standard basis of $\Z^s \times 0_{\Z^r}$,
$(v_1, \ldots, v_r)$ the standard basis of $0_{\Z^s} \times \Z^r$,
$$
v_0 = -(v_1 + \ldots + v_r) \quad \text{and} \quad
w_0 = a_1 v_1 + \ldots + a_r v_r -(w_1 + \ldots + w_s).
$$
We denote by $D_{v_i}$ the divisor corresponding to the ray $\Cone(v_i)$ and $D_{w_j}$
the divisor corresponding to the ray $\Cone(w_j)$.
We have the following linear equivalence,
\begin{equation}\label{eq:divisor-toricpicard-two}
\left\lbrace
\begin{array}{ll}
D_{v_i} \sim_\lin D_{v_0} - a_i D_{w_0} & \text{for} ~ i \in \{1, \ldots, r\} \\
D_{w_j} \sim_\lin D_{w_0} & \text{for} ~ j \in \{1, \ldots, s \}
\end{array}
\right. .
\end{equation}
By (\ref{eq:divisor-toricpicard-two}), we deduce that $\Pic(X)$ is generated by
$D_{v_0}$ and $D_{w_0}$.

\begin{prop}[{\cite[Proposition 4.2.1]{DDK20}}]
Let $D = \alpha D_{w_0} + \beta D_{v_0}$ be an invariant divisor of $X$ with
$\alpha, \beta \in \Z$. Then, $D$ is ample if and only if $\alpha >0$ and $\beta >0$.
\end{prop}

By Theorem \ref{theo:canonical-divisor}, the anti-canonical divisor of $X$ is given by
\begin{equation}\label{eq:canonical-toricpicard-two}
- K_X = \sum_{i=0}^{r}{D_{v_i}} + \sum_{j=0}^{s}{D_{w_j}} \sim_\lin
(s+1 - a_1 - \ldots - a_r) D_{w_0} + (r+1) D_{v_0}.
\end{equation}
Thus, $X$ is a Fano variety if and only if $a_1 + \ldots + a_r \leq s$.

\begin{rem}
For $\alpha, \beta \in \N^\ast$ and $L = \cO_{X}(\alpha D_{w_0} + \beta D_{v_0})$, we have
an isomorphism $L \cong \pi^{\ast} \cO_{\bP^s}(\alpha) \otimes \cO_{X}(\beta)$.
\end{rem}

Let $L= \pi^{\ast} \cO_{\bP^s}(\nu) \otimes \cO_{X}(1)$ be an ample $\Q$-divisor of $X$
with $\nu \in \Q_{>0}$. For $k \in \{1, \ldots, s\}$, we set
$\Delta_k = \Conv(0, w_1, \ldots, w_k)$.
By \cite[Section 4]{HNS19}, the polytope corresponding to the $\Q$-polarized toric
variety $(X, L)$ is given by
$$
P = \Conv \left( \nu \Delta_s \times \{0 \} \cup (a_1 + \nu) \Delta_s \times \{v_1 \}
\cup \ldots \cup (a_r + \nu) \Delta_s \times \{v_r \} \right) .
$$
We denote by $P^{v_i}$ (resp. $P^{w_j}$) the facet of $P$ corresponding to the ray
$\Cone(v_i)$ (resp. $\Cone(w_j)$). The facet $P^{v_i}$ is the convex hull of
\begin{align*}
\nu \Delta_s \times \{0 \} \cup  \ldots
& \cup (a_{i-1} + \nu) \Delta_s \times \{v_{i-1} \} \\
& \cup (a_{i+1} + \nu) \Delta_s \times
\{v_{i+1} \} \cup \ldots \cup (a_r + \nu) \Delta_s \times \{v_r\}
\end{align*}
and $P^{w_i}$ is isomorphic to
$$
\nu \Delta_{s-1} \times \{0 \} \cup (a_1 + \nu) \Delta_{s-1} \times \{v_1 \}
\cup \ldots \cup (a_r + \nu) \Delta_{s-1} \times \{v_r \} ~.
$$
By \cite[Proposition 4.3]{HNS19}, for any $j \in \{0, \ldots, s \}$,
$$
\vol(P^{w_j}) = \sum_{k=0}^{s-1}{ \dbinom{s+r-1}{k} \left(
\sum_{d_1 + \ldots + d_r =s-k-1}{ a_{1}^{d_1} \cdots a_{r}^{d_r} } \right) \nu^k }
$$
and
$$
\vol(P^{v_0}) = \sum_{k=0}^{s}{ \dbinom{s+r-1}{k} \left( \sum_{d_1 + \ldots + d_r =s-k}{
a_{1}^{d_1} \cdots a_{r}^{d_r} } \right) \nu^k } ~.
$$
If $i \in \{1, \ldots, r \}$, we have
$$
\vol(P^{v_i}) = \sum_{k=0}^{s}{ \dbinom{s+r-1}{k} \left(
\sum_{\substack{d_1 + \ldots + d_{i-1} \\ + d_{i+1} + \ldots + d_r=s-k}}{
a_{1}^{d_1} \cdots a_{i-1}^{d_{i-1}} a_{i+1}^{d_{i+1}} \cdots a_{r}^{d_r} }
\right) \nu^k } .
$$
All these formulas will be used from Section \ref{sec:stability-logtangent-toricpic2}
when we study the stability of logarithmic tangent sheaves on toric varieties of Picard
rank two.

\subsection{Equivariant reflexive sheaves and families of filtrations}
\label{sec:equiv-sheaves}
Let $X$ be a toric variety associated to a fan $\Sigma$ in $N_\R$. Recall that a
reflexive sheaf on $X$ is a coherent sheaf $\cE$ that is canonically isomorphic to its
double dual $\cE^{\vee\vee}$.

Let $\theta : T \times X \rightarrow X$ be the action of $T$ on $X$,
$\mu : T \times T \rightarrow T$ the group multiplication,
$p_2 : T \times X \rightarrow X$ the projection onto the second factor and
$p_{23} : T \times T \times X \rightarrow T \times X$ the projection onto the second and
the third factor.
We call a sheaf $\cE$ on $X$ {\it equivariant} if there exists an isomorphism
$\Phi : \theta^{\ast} \cE \rightarrow p_{2}^{\ast} \cE$ such that
\begin{equation}\label{eq:cocycle-equiv-sheaf}
(\mu \times \Id_X)^{\ast} \Phi = p_{23}^{\ast} \Phi \circ
(\Id_T \times \theta)^{\ast} \Phi .
\end{equation}

Klyachko gave a description of torus equivariant reflexive sheaves over toric varieties
in terms of combinatorial data \cite{Kly90}:

\begin{defn}
\label{def:family of filtrations}
A family of filtrations $\E$ is the data of a finite dimensional vector space $E$ and for
each ray $\rho\in\Sigma(1)$, an increasing filtration $(E^\rho(i))_{i\in\Z}$ of $E$ such
that $E^\rho(i)=\lbrace 0 \rbrace$ for $i\ll 0$ and $E^\rho(i)=E$ for some $i$.
\end{defn}

\begin{rem}
Note that we are using increasing filtrations here, as in \cite{Per04}, rather than
decreasing as in \cite{Kly90}.
\end{rem}

To a family of filtrations
$\E:= \left( E, \{ E^{\rho}(j) \}_{\rho \in \Sigma(1), \, j \in \Z} \right)$,
we can assign an equivariant reflexive sheaf $\cE:=\fK(\E)$ defined by
\begin{equation}
\label{eq:sheaf from family of filtrations}
\Gamma(U_{\sigma}, \cE):=\bigoplus_{m\in M} \bigcap_{\rho\in\sigma(1)}
E^\rho(\langle m,u_\rho\rangle)\otimes \chi^m
\end{equation}
for all positive dimensional cones $\sigma\in\Sigma$, while
$\Gamma(U_{\lbrace 0\rbrace},\cE)=E\otimes \mathbb{C}[M]$. The morphisms between families
of filtrations are linear maps preserving the filtrations.
Then, by \cite[Theorem 5.19]{Per04}, the functor $\fK$ induces an equivalence of
categories between the families of filtrations and equivariant reflexive sheaves over $X$.

\begin{notation}\label{nota:decomposition-of-epuiv-sheaf}
Let $\cE$ be an equivariant reflexive sheaf given by the family of filtrations
$\left( E, \{ E^{\rho}(j) \}_{\rho \in \Sigma(1), \, j \in \Z} \right)$.
For any $\rho \in \Sigma$, we denote the space $\Gamma(U_\rho,\cE)$ by $E^\rho$ and
we write
$$E^\rho = \bigoplus_{m \in M}{E_{m}^\rho \otimes \chi^m}$$
where for any $m \in M$, $E_{m}^{\rho} := E^{\rho}(\< m, u_\rho \>)$.
\end{notation}

\begin{example}[Tangent sheaf {\cite[Corollary 2.2.17]{DDK20}}]
\label{examp:filtration-tangent}
The family of filtrations of the tangent sheaf $\cT_X$ of $X$ is given by
$$
E^{\rho}(j) = \left\lbrace
\begin{array}{ll}
0 & \text{if}~ j< -1 \\
\Span(u_{\rho}) & \text{if}~ j = -1 \\
N \otimes_{\Z} \C & \text{if} ~ j> -1
\end{array}
\right..
$$
\end{example}

\subsection{Some stability notions}
%In this part, we are interested in the notion of slope stability. We refer to the paper
%of Takemoto \cite{Tak72} for the definitions.
We denote by $\Amp(X) \subset N^1(X) \otimes_{\Z} \R$ the {\it ample cone} of $X$.
Let $\cE$ be a torsion-free coherent sheaf on $X$. The {\it degree} of $\cE$ with respect
to an ample class $L \in \Amp(X)$ is the real number obtained by intersection
$$
\deg_{L}(\cE)= c_1(\cE) \cdot L^{n-1}
$$
and its {\it slope} with respect to $L$ is given by
$$
\mu_{L}(\cE) = \dfrac{ \deg_{L}(\cE) }{\rk(\cE) }.
$$

\begin{defn}
A torsion-free coherent sheaf $\cE$ is said to be {\it slope semistable}
(or {\it semistable} for short) with respect to $L \in \Amp(X)$ if for any proper coherent
subsheaf of lower rank $\cF$ of $\cE$, one has
$$
\mu_{L}(\cF) \leq \mu_{L}(\cE) .
$$
When strict inequality always holds, we say that $\cE$ is {\it stable}. Finally, $\cE$ is
said to be {\it polystable} if it is the direct sum of stable subsheaves of the same slope.
\end{defn}

\begin{prop}[{\cite[Claim 2 of Proposition 4.13]{Koo11}}]
\label{prop:poly-vs-semistability}
A reflexive polystable sheaf on $X$ is a semistable sheaf on $X$ isomorphic to a
(finite, nontrivial) direct sum of reflexive stable sheaves.
Let $\cE$ be a semistable reflexive sheaf on $X$, then $\cE$ contains a unique maximal
reflexive polystable subsheaf of the same slope as $\cE$.
\end{prop}

If $\cE$ is an equivariant reflexive sheaf on a normal toric variety $X$ given by the
family of filtrations $\left( E, \{ E^{\rho}(j) \} \right)$, according to
\cite[Proposition 4.13]{Koo11}, it is enough to test slope inequalities for equivariant
and reflexive saturated subsheaves. By \cite[Proposition 2.3]{HNS19}, if $\cF$ is an
equivariant reflexive subsheaf of $\cE$ given by the family of filtrations
$(F, \{ F^\rho(i) \})$ with $F$ a vector subspace of $E$ and
$F^{\rho}(i) \subseteq E^{\rho}(i)$, then $\cF$ is saturated in $\cE$ if and only if for
all $\rho \in \Sigma(1), i \in \Z$, $F^\rho(i) = E^\rho(i) \cap F.$

\begin{notation}\label{nota:saturated-subsheaf}
Let $F$ be a vector subspace of $E$. We denote by $\cE_F$ the saturated subsheaf of
$\cE$ defined by the family of filtrations $\left( F, \{ F^{\rho}(j) \} \right)$
where $F^{\rho}(j)= F \cap E^{\rho}(j).$
\end{notation}

By \cite[Corollary 3.18]{Koo11}, the first Chern class of $\cE$ is given by
\begin{equation}\label{eq:c1-sheaf}
c_1(\cE) = - \sum_{\rho \in \Sigma(1)} e^{\rho}(\cE) \, D_{\rho}
\quad \text{where} \quad
e^{\rho}(\cE) = \sum_{i \in \Z}{i \, e^{\rho}(i)}
\end{equation}
with $e^{\rho}(i) = \dim E^{\rho}(i) - \dim E^{\rho}(i-1) \,$. Therefore, for any
$L \in \Amp(X)$,
\begin{equation}\label{eq:slope-sheaf}
\mu_{L}(\cE) = - \dfrac{1}{\rk(\cE)} \sum_{\rho \in \Sigma(1)}{ e^{\rho}(\cE)
\deg_{L}(D_{\rho})} ~.
\end{equation}
For a reflexive sheaf $\cE$ on $X$, we set
\begin{align*}
\Stab(\cE) & = \{L \in \Amp(X): \cE ~\text{is stable with respect to $L$} \}
\quad \text{and} \\
\sStab(\cE) & = \{L \in \Amp(X): \cE ~\text{is semistable with respect to $L$} \} ~.
\end{align*}

\section{Description of equivariant logarithmic tangent sheaves}
\label{sec:log-sheaves}

\subsection{Logarithmic tangent sheaves}\label{sec:logarithmic-sheaves}
We recall here the definition of the logarithmic tangent sheaf of a pair $(X, D)$
where $X$ is a normal projective variety of dimension $n$ and $D$ a reduced Weil divisor
on $X$.

\begin{defn}
We say that a pair $(X, D)$ is {\it log-smooth} if $X$ is smooth and $D$ is a reduced snc
divisor. We denote by $(X, D)_\reg$ the snc locus of the pair $(X,D)$, that is, the locus
of points $x \in X$ where $(X,D)$ is log-smooth in a neighborhood of $x$.
\end{defn}

If the pair $(X,D)$ is log-smooth, we define the logarithmic tangent bundle\linebreak
$T_{X}(-\log D)$ as the dual of the bundle of logarithmic differential form
$\Omega_{X}^{1}(\log D)$ where $\Omega_{X}^{1}(\log D)$ is defined in \cite[\S 1]{Ita76}.
By \cite[Definition 4]{Kaw78} and \cite[\S 1]{Sai75}, we can see the space of sections of
$T_{X}(- \log D)$ as the set of vector fields on $X$ which vanish along $D$. If $D$ is
locally given by $(z_1 \cdots z_k = 0)$, then $T_{X}(- \log D)$ as a sheaf is the locally
free $\cO_X$-module generated by
$$
z_1 \dfrac{\partial}{\partial z_1} , \ldots, z_k \dfrac{\partial}{\partial z_k} ,
\dfrac{\partial}{\partial z_{k+1}}, \ldots, \dfrac{\partial}{\partial z_n}.
$$

\begin{defn}[{\cite[Definition 3.4]{Gue16}}]
Let $(X,D)$ be a log pair and $X_0 = (X, D)_\reg$. The logarithmic tangent sheaf
$\cT_{X}(- \log D)$ of $(X, D)$ is defined as $j_{\ast} T_{X_0}(- \log D_{| X_0})$ where
$j : X_0 \rightarrow X$ is the open immersion.
\end{defn}

The sheaf $\cT_{X}(- \log D)$ (as well as its dual) is coherent; by
\cite[Proposition 1.6]{Har80}, this sheaf is reflexive.
We now consider the case where $X$ is a toric variety with torus $T$.
Let $\Sigma$ be the fan of $X$ and $X_0$ the toric variety corresponding to the fan
$\Sigma^1 = \Sigma(0) \cup \Sigma(1)$. We denote by $j : X_0 \rightarrow X$ the open
immersion.

\begin{prop}\label{prop:equivariant-log-sheaves}
Let $D$ be a reduced Weil divisor on $X$. The sheaf $\Omega_{X}^{1}(\log D)$
is equivariant if and only if $D$ is an invariant divisor under the torus action.
\end{prop}

\begin{proof}
We assume that $D$ is an invariant divisor under the torus action.
Let $D_0$ be the restriction of $D$ on $X_0$. For $t \in T$, let $\phi_t : X \rightarrow X$
be the map defined by $\phi_t(x) = t \cdot x$. We set $\Phi_t = ( d \phi_t )^{-1}$
where $d \phi_t$ is the differential of $\phi_t$.
If $\cE = T_{X_0}$, we get the following diagram.
\begin{equation}\label{eq:cocycle-equiv-sheaf2}
\begin{tikzcd}
(\phi_{t' \cdot t})^{\ast} \cE \arrow[rr, "\Phi_{t' \cdot t}"]
\arrow[rd, "\phi_{t}^{\ast} \Phi_{t'}"'] & & \cE \\
& \phi_{t}^{\ast} \cE \arrow[ru, "\Phi_t"'] &
\end{tikzcd}
\end{equation}
If $\cE = T_{X_0}(-\log D_0)$, the diagram (\ref{eq:cocycle-equiv-sheaf2}) remains true;
thus, $T_{X_0}(- \log D_0)$ is equivariant.
Therefore $\Omega_{X_0}^{1}(\log D_0)$ is equivariant. As
\begin{equation}\label{eq:log-cotangent}
\Omega_{X}^{1}(\log D) \cong j_{\ast} \Omega_{X_0}^{1}(\log D_0)~,
\end{equation}
we deduce that $\Omega_{X}^{1}( \log D)$ is equivariant.

We now assume that $\Omega_{X}^{1}(\log D)$ is equivariant. We write
$D = \sum_{k=1}^{s}{D_k}$ where the $D_k$ are irreducible Weil divisors of $X$.
\\
{\it First case.}
We assume that $X$ is smooth. By \cite[Properties 2.3]{EV92} we have an exact sequence
$$
0 \imp \Omega_{X}^1 \imp \Omega_{X}^{1}\left( \log D \right) \imp
\bigoplus_{k=1}^{s} \cO_{D_k} \imp 0
$$
where $\cO_{D_k}$ is viewing as a sheaf on $X$ via extension by zero.
The first part of the proof is to show that : for any $t \in T$, $t \cdot Z = Z$ where
$Z = X \setminus D$. Let $x \in Z$ and assume that there is $t \in T$ such that
$y  = t \cdot x \in D$. We have two exact sequences
\begin{align*}
0 \imp \Omega_{X, \, x}^1 \imp \Omega_{X}^{1}\left( \log D \right)_x \imp
\bigoplus_{k=1}^{s} \cO_{D_k, x} \imp 0
\\
0 \imp \Omega_{X, \, y}^1 \imp \Omega_{X}^{1}\left( \log D \right)_y \imp
\bigoplus_{k=1}^{s} \cO_{D_k, y} \imp 0
\end{align*}
As $\Omega_{X}^{1}$ and $\Omega_{X}^{1}(\log D)$ are equivariant, we have an isomorphism
$$
\bigoplus_{k=1}^{s} \cO_{D_k, x} \cong \bigoplus_{k=1}^{s} \cO_{D_k, y} ~;
$$
this is absurd. Therefore, for any $t \in T$, we have $t \cdot Z \subseteq Z$, that is
$t \cdot Z = Z$. As $\Omega_{X}^{1}(\log D)$ is equivariant, by using the fact that
$D = X \setminus Z$, for any $t \in T$, we have $t \cdot D = D \,$; thus, $D$ is a
$T$-invariant divisor.
\\
{\it Second case.}
We assume that $X$ is a normal variety. By (\ref{eq:log-cotangent}),
as $\Omega_{X}^{1}(\log D)$ is equivariant, we also have the same property for
$\Omega_{X_0}^{1}(\log D_0)$. By the first case, $D_0$ is an invariant divisor under the
action of $T$ on $X_0$. As $\codim(X \setminus X_0) \geq 2$, we deduce that $D$ is the
Zariski closure of $D_0$ on $X$. Thus, $D$ is an invariant divisor under the action of
$T$ on $X$.
\end{proof}

\subsection{Families of filtrations of logarithmic tangent sheaves}
\label{sec:filtration-logtangent}
We give here the proof of Theorem \ref{theo:filt-logtangent-intro}.
Let $X$ be a toric variety of dimension $n$ associated to the fan $\Sigma$ and $D$ a
reduced Weil divisor of $X$. According to Proposition \ref{prop:equivariant-log-sheaves},
$\cT_{X}(- \log D)$ is equivariant if and only if
$$
D = \sum_{\rho \in \Delta}{ D_{\rho}}
$$
where $\Delta \subseteq \Sigma(1)$.
In that case, $\cE$ is given by a family of filtrations.

\begin{rem}\label{rem:action-to-ring}
If $G$ is an algebraic group acting on the affine toric variety $Y= \Spec(R)$, we define
an action of $G$ on $R$ by setting : for any $g \in G$ and $\varphi \in R$,
$g \cdot \varphi = \left( \phi_{g^{-1}} \right)^{\ast} \varphi$
where $\phi_g(x) = g \cdot x$.
\end{rem}

\begin{proof}[\bf Proof of Theorem \ref{theo:filt-logtangent-intro}]
For $\rho \in \Sigma(1)$, we set
$$E^{\rho} = \Gamma(U_{\rho}, \cT_{X}(- \log D)) .$$
By the orbit-cone correspondence (cf. Theorem \ref{theo:orbit-cone}),
if $\rho \in \Delta$, we have $U_{\rho} \cap D = U_{\rho} \cap D_{\rho}$ and for
$\rho \notin \Delta$, $U_{\rho} \cap D = \nothing$. We can reduce the problem to the case
where $\Delta$ contains one ray.
For the rest of the proof, we assume that $\Delta = \{ \rho_0 \}$.
Let $\rho \in \Sigma(1)$ and $(u_1, \ldots, u_n)$ a basis of $N$ such that
$u_1 = u_{\rho}$. We denote by $(e_1, \ldots, e_n)$ the dual basis of
$(u_1, \ldots, u_n)$ and we set $x_i = \chi^{e_i}$. We have
$\C[S_{\rho}] = \C[x_1, x_{2}^{\pm 1}, \ldots, x_{n}^{\pm 1}]$.
\\
{\bf First case : We assume that $\rho = \rho_0$.}
As on $U_{\rho}$ the divisor $D$ is defined by the equation $x_{1}=0$, we have
$$
E^{\rho} = \left( \C[S_{\rho}] \cdot x_1 \dfrac{\partial}{\partial x_1} \right)
\oplus
\left( \bigoplus_{i=2}^{n} \C[S_{\rho}] \cdot \dfrac{\partial}{\partial x_i} \right) ~.
$$
We set
$$
L_{1}^{\rho} = \bigoplus_{m \in S_{\rho}}{
\C \cdot \chi^{m+e_1} \dfrac{\partial}{\partial x_1} }
\quad
\text{and for $i \in \{2, \ldots, n \}$},
\quad
L_{i}^{\rho} = \bigoplus_{m \in S_{\rho}}{
\C \cdot \chi^m \dfrac{\partial}{\partial x_i} } ~~.
$$
According to Remark \ref{rem:action-to-ring}, for any $t \in T$ and $m \in M$,
$t \cdot \chi^m = \chi^{-m}(t) \chi^m$. Hence, $t \cdot d x_i = \chi^{-e_i}(t) d x_i$ and
$t \cdot \dfrac{\partial}{\partial x_i} = \chi^{e_i}(t) \dfrac{\partial}{\partial x_i}$.
For $i \in \{1, \ldots, n \}$, we write
$$
L_{i}^{\rho} = \bigoplus_{m \in M}{ \left( L_{i}^{\rho} \right)_m} \quad
\text{where} \quad
\left( L_{i}^{\rho} \right)_m = \{ f \in L_{i}^{\rho} :
t \cdot f = \chi^{-m}(t) f \} ~.
$$
We have
$$
\left( L_{1}^{\rho} \right)_m = \left\lbrace
\begin{array}{ll}
\C \cdot \chi^{m + e_1} \, \dfrac{\partial}{\partial x_1} &
\text{if}~ 0 \preceq_{\rho} m
\\ 0 & \text{otherwise}
\end{array}
\right.
$$
and for $i \in \{2, \ldots, \, n \}$,
$$
\left( L_{i}^{\rho} \right)_m = \left\lbrace
\begin{array}{ll}
\C \cdot \chi^{m + e_i} \, \dfrac{\partial}{\partial x_i} &
\text{if}~  -e_i \preceq_{\rho} m
\\ 0 & \text{otherwise}
\end{array}
\right. ~.
$$
As the torus $T$ is a Lie group, the tangent space of $T$ at the identity element
generated by $\left(\frac{\partial}{\partial x_i} \right)_{1 \leq i \leq n}$ is
isomorphic to $N_{\C}$. Thus, for all $i \in \{1, \ldots, n \}$, we can identify
$\frac{\partial}{\partial x_i}$ with $u_i$.
\\
For $i \in \{1, \ldots, n \}$, we set $\bL_{i}^{\rho} = \Span(u_i)$. Let $m \in M$.
\begin{itemize}
\item
If $i = 1$ and $0 \preceq_{\rho} m$, then $\left( L_{i}^{\rho} \right)_m$
is isomorphic to $\bL_{1}^{\rho} \otimes \chi^m$.
\item
If $i \geq 2$ and $-e_i \preceq_{\rho} m$, then $\left( L_{i}^{\rho} \right)_m$
is isomorphic to $\bL_{i}^{\rho} \otimes \chi^m$.
\end{itemize}
We set $j = \langle m, u_1 \rangle$. The condition $0 \preceq_{\rho} m$
is equivalent to $j \geq 0$ and for $i \in \{2, \ldots, \, n \}$, $-e_i \preceq_{\rho} m$
is equivalent to $j \geq 0$. Thus, for any $i \in \{1, \ldots, n \}$, we set
$$
L_{i}^{\rho}(j) = \left\lbrace
\begin{array}{ll}
0 & \text{if} ~ j \leq -1 \\ \bL_{i}^{\rho} & \text{if} ~ j \geq 0
\end{array}
\right. .
$$
By construction, $\{L_{i}^{\rho}(j)\}$ is the family of filtrations of $L_{i}^{\rho}$.
As
$$
E^{\rho} = \bigoplus_{m \in M}{E^{\rho}(\langle m, u_1 \rangle) \otimes \chi^m}
$$
where
$E^{\rho}(\< m, u_1 \>) \cong \bigoplus_{i=1}^{n}{ L_{i}^{\rho}(\< m, u_{\rho} \>) }$,
we get
$$
E^{\rho}(j) \cong \left\lbrace
\begin{array}{ll}
0 & \text{if}~ j \leq -1 \\ N_{\C} & \text{if}~ j \geq 0
\end{array}
\right. ~.$$
{\bf Second case : We assume that $\rho \neq \rho_0$.}
As $U_{\rho} \cap D = \nothing$, we have
$$
E^{\rho} = \bigoplus_{i=1}^{n} \C[S_{\rho}] \cdot \dfrac{\partial}{\partial x_i}
=
\bigoplus_{i=1}^{n} \left( \bigoplus_{m \in S_{\rho}}{
\C \cdot \chi^m \, \dfrac{\partial}{\partial x_i} } \right) ~~.
$$
For all $i \in \{1, \ldots, n \} \,$, we set
$L_{i}^{\rho} = \C[S_{\rho}] \cdot \dfrac{\partial}{\partial x_i}$.
We have
$$
L_{i}^{\rho} = \bigoplus_{m \in M}{ \left( L_{i}^{\rho} \right)_m}
\quad \text{where} \quad
\left( L_{i}^{\rho} \right)_m = \left\lbrace
\begin{array}{ll}
\C \cdot \chi^{m + e_i} \, \dfrac{\partial}{\partial x_i} &
\text{if}~ -e_i \preceq_{\rho} m
\\ 0 & \text{oherwise}
\end{array}
\right.~.
$$
For $m \in M$, we set $j = \langle m, \, u_1 \rangle$.
The condition $- e_i \preceq_{\rho} m$ is equivalent to
$j \geq - \langle e_i, \, u_1 \rangle$. Thus, for any $i \in \{2, \ldots, n \}$, the
family of filtrations of $L_{i}^{\rho}$ is given by
$$
L_{i}^{\rho}(j) = \left\lbrace
\begin{array}{ll}
0 & \text{if} ~ j \leq -1 \\ \bL_{i}^{\rho} & \text{if} ~ j \geq 0
\end{array}
\right.
$$
and the family of filtrations of $L_{1}^{\rho}$ is given by
$$
L_{1}^{\rho}(j) = \left\lbrace
\begin{array}{ll}
0 & \text{if} ~ j \leq -2 \\ \bL_{i}^{\rho} & \text{if} ~ j \geq -1
\end{array}
\right. ~.
$$
As in the first case, we get
$$
E^{\rho}(j) \cong \left\lbrace
\begin{array}{ll}
0 & \text{if} ~ j \leq -2 \\
\Span(u_{\rho}) & \text{if} ~ j = -1 \\
N \otimes_{\Z} \C & \text{if} ~ j \geq 0
\end{array}
\right.$$
which ends the proof.
\end{proof}

The sheaf of regular sections of the trivial vector bundle $X \times \C \rightarrow X$ of
rank $1$ is $\cO_{X}$. We denote by
$\left(F, \, \{ F^{\rho}(j) \}_{\rho \in \Sigma(1), \, j \in \Z} \right)$
the family of filtration of $\cO_X$.
For $\rho \in \Sigma(1)$, we set $F^{\rho} = \cO_{X}(U_{\rho})$ and
$F_{m}^{\rho} = \{ f \in F^{\rho} : t \cdot f = \chi^{-m}(t) f \}$. As
$$
F^{\rho} = \bigoplus_{m \in M}{ F_{m}^{\rho}} = \C[S_{\rho}] =
\bigoplus_{m \in S_{\rho}}{\C \cdot \chi^m} ,
$$
we deduce that $F_{m}^{\rho} = \C \cdot \chi^m$ if $m \in S_{\rho}$ and
$F_{m}^{\rho} = 0$ if $m \notin S_{\rho}$. Hence,
$$
F^{\rho}(j) = \left\lbrace
\begin{array}{ll}
0 & \text{if}~ j \leq -1 \\ \C & \text{if}~ j \geq 0
\end{array}
\right. ~.
$$

\begin{cor}\label{cor:filtration-logtangent}
Let $\Delta \subseteq \Sigma(1)$ and $D = \sum_{\rho \in \Delta}{ D_{\rho}}$.
\begin{enumerate}
\item
If $\Delta = \nothing$, then $\cT_{X}(- \log D)$ is the tangent sheaf $\cT_{X}$.
\item
If $\Delta = \Sigma(1)$, then $\cT_{X}(- \log D)$ is isomorphic to the trivial sheaf
of rank $n$.
\end{enumerate}
\end{cor}

\begin{proof}
If $\Delta= \nothing$, the family of filtrations of $\cT_{X}(- \log D)$ is identical
to the family of filtrations given in Example \ref{examp:filtration-tangent}.
If $\Delta = \Sigma(1)$, for all $\rho \in \Sigma(1)$, we have
$$
E^{\rho}(j) = \left\lbrace
\begin{array}{ll}
0 & \text{if} ~~ j \leq -1 \\ N \otimes_{\Z} \C & \text{if} ~~ j \geq 0
\end{array}
\right. ~~.
$$
Hence, $\cT_{X}(- \log D)$ is isomorphic to the trivial sheaf of rank $n$.
\end{proof}

\begin{notation}\label{nota:saturated-logtangent}
Let $G$ be a vector subspace of $N_{\C}$. We denote by $\cE_G$ the subsheaf of
$\cE = \cT_{X}(- \log D)$ defined by the family of filtrations
$\left( E_G, \{ G^{\rho}(j) \}_{\rho \in \Sigma(1), \, j \in \Z} \right)$
where $E_G = G$ and $G^{\rho}(j) = E^{\rho}(j) \cap G$.
If $\rho \in \Delta$ or $u_{\rho} \notin G$, then
$$
G^{\rho}(j) = \left\lbrace
\begin{array}{ll}
0 & \text{if}~ j \leq -1 \\ G & \text{if}~ j \geq 0
\end{array}
\right. ~.
$$
If $\rho \notin \Delta$ and $u_{\rho} \in G$, then
$$
G^{\rho}(j) = \left\lbrace
\begin{array}{ll}
0 & \text{if}~ j \leq -2 \\
\Span(u_{\rho}) & \text{if}~ j = -1 \\
G & \text{if}~ j \geq 0
\end{array}
\right. ~.
$$
\end{notation}

\subsection{Decomposition of equivariant logarithmic tangent sheaves}
\label{sec:decomposition-of-logtangent}
In this part, we give some conditions on $\Sigma$ and $\Delta$ which ensure that the
logarithmic tangent sheaf is decomposable. We first recall the family of filtrations
of a direct sum of equivariant reflexive sheaves.

\begin{prop}[{\cite[Section 6.3]{IS15}}]
Let $\cF$ and $\cG$ be two equivariant reflexive sheaves with
$\left(F, \, \{ F^{\rho}(j) \}_{\rho \in \Sigma(1), \, j \in \Z} \right)$ and
$\left(G, \, \{ G^{\rho}(j) \}_{\rho \in \Sigma(1), \, j \in \Z} \right)$ for
family of filtrations. The family of filtrations of $\cF \oplus \cG$ is given by
\begin{equation}\label{eq:filtration-sumsheaves}
\left(F \oplus G, \, \{ (F \oplus G)^{\rho}(j) \}_{\rho \in \Sigma(1), \, j \in \Z}
\right)
\quad \text{where} \quad (F \oplus G)^{\rho}(j) = F^{\rho}(j) \oplus G^{\rho}(j) .
\end{equation}
\end{prop}

We assume that $X$ is a toric variety without torus factor. We denote by $p$ the rank of
the class group $\Cl(X)$ of $X$. By Corollary \ref{cor:number-rays}, we have
$\card(\Sigma(1)) = n + p$.

\begin{prop}\label{prop:logtangent-notsimple1}
Let $D= \sum_{\rho \in \Delta} D_\rho$ with $\card(\Delta) = p$. We set
$\Sigma(1) \setminus \Delta = \{ \rho_1, \ldots, \rho_n \}$ where $\rho_k = \Cone(u_k)$
and $u_k \in N$. If $N_{\R} = \Span(u_1, \ldots, u_n)$, then $\cE = \cT_{X}(- \log D)$
is decomposable and
$$
\cE = \bigoplus_{k=1}^{n}{ \cE_{F_k}}
$$
where $\cE_{F_k}$ is the subsheaf of $\cE$ corresponding to the vector space
$F_k = \Span(u_k)$.
\end{prop}

\begin{proof}
For all $k \in \{1, \ldots, n \}$, the family of filtrations
$\left( F_k, \, \{ F_{k}^{\rho}(j) \} \right)$
of $\cE_{F_k}$ is given by
$$
F_{k}^{\rho}(j) = \left\lbrace
\begin{array}{ll}
0 & \text{if}~ j \leq -1 \\ F_k & \text{if}~ j \geq 0
\end{array}
\right.~ \text{if}~ \rho \neq \Cone(u_k)
$$
and
$$
F_{k}^{\rho}(j) = \left\lbrace
\begin{array}{ll}
0 & \text{if}~ j \leq -2 \\
\Span(u_{\rho}) & \text{if}~ j = -1 \\
F_k & \text{if}~ j \geq 0
\end{array}
\right.~ \text{if}~ \rho = \Cone(u_k) ~.
$$
For all $\rho \in \Sigma(1)$ and $j \in \Z$, we have
$$
\bigoplus_{k =1}^{n}{ F_{k}^{\rho}(j) } = \left\lbrace
\begin{array}{ll}
0 & \text{if}~ j \leq -1 \\ N_{\C} & \text{if}~ j \geq 0
\end{array}
\right.
\quad \text{if}~ \rho \in \Delta
$$
and
$$
\bigoplus_{k =1}^{n}{ F_{k}^{\rho}(j) } = \left\lbrace
\begin{array}{ll}
0 & \text{if}~ j \leq -2 \\
\Span(u_{\rho}) & \text{if}~ j = -1 \\
N_{\C} & \text{if}~ j \geq 0
\end{array}
\right.
\quad \text{if}~ \rho \notin \Delta ~~.
$$
Hence, by (\ref{eq:filtration-sumsheaves}) and Theorem \ref{theo:filt-logtangent-intro}
we get $\cE = \bigoplus_{k =1}^{n}{ \cE_{F_k}}$.
\end{proof}

A similar proof gives the following result.

\begin{prop}\label{prop:logtangent-notsimple2}
We assume that $\Delta$ satisfies $1 + p \leq \card(\Delta) \leq n+p-1$. Then the sheaf
$\cE = \cT_{X}(- \log D)$ is decomposable and $\cE = \cE_{G} \oplus \cE_F$ where
$G = \Span(u_{\rho} : \rho \in \Sigma(1) \setminus \Delta)$ and $F$ a vector subspace
of $N_{\C}$ such that $N_{\C} = G \oplus F$.
\end{prop}

\subsection{An instability condition for logarithmic tangent sheaves}
\label{sec:instability-condition}

Let $\Delta \subseteq \Sigma(1)$ and $D = \sum_{\rho \in \Delta}{D_{\rho}}.$
Let $\left( E_G, \{ G^{\rho}(j) \}_{\rho \in \Sigma(1), \, j \in \Z} \right)$ be the
family of filtrations corresponding to the subsheaf $\cE_G$ of
$\cE = \cT_{X}(-\log D)$ where $G \subseteq N_\C$ is a vector subspace.
By Equation (\ref{eq:slope-sheaf}), if $L$ is a polarization of $X$, we have
\begin{equation}\label{eq:slope-logtangent}
\mu_{L}(\cE) = \dfrac{1}{n} \sum_{\rho \notin \Delta}{ \deg_{L}(D_{\rho})}
\end{equation}
and
\begin{equation}\label{eq:slope-logtangent2}
\mu_{L}(\cE_G) = \dfrac{1}{\dim G} \sum_{\rho \notin \Delta \; \text{and} \;
u_{\rho} \in G}{ \deg_{L}(D_{\rho})} ~.
\end{equation}
Therefore,
\begin{equation}\label{eq:stability-logtangent}
\mu_L(\cE) - \mu_{L}(\cE_G) = \left( \dfrac{1}{n} - \dfrac{1}{\dim G} \right)
\sum_{\substack{\rho \notin \Delta \\ u_{\rho} \in G}}{ \deg_{L}(D_{\rho})}
+ \dfrac{1}{n}
\sum_{\substack{\rho \notin \Delta \\ u_{\rho} \notin G}}{ \deg_{L}(D_{\rho})}.
\end{equation}
To study the stability of $\cE$ with respect to $L \in \Amp(X)$, it suffices to
compare $\mu_{L}(\cE)$ with $\mu_{L}(\cE_G)$ where
$G \subseteq \Span( u_{\rho} : \rho \notin \Delta)$ and $1 \leq \dim G \leq n-1 $.

\begin{prop}\label{prop:logtangent-unstability}
If $1 \leq \card(\Sigma(1) \setminus \Delta) \leq n-1$, then for any $L \in \Amp(X)$,
the logarithmic tangent sheaf $\cE = \cT_{X}(- \log D)$ is not semistable with
respect to $L$.
\end{prop}

\begin{proof}
We assume that $\Sigma(1) \setminus \Delta = \{\rho_1, \ldots, \, \rho_k \}$ where
$1 \leq k \leq n-1$ and we denote by $D_j$ the divisor corresponding to
$\rho_j = \Cone(u_j)$. For $G = \Span(u_1, \ldots, u_k)$, we have
$$
\mu_{L}(\cE) - \mu_{L}(\cE_G) = \left( \dfrac{1}{n} - \dfrac{1}{\dim G} \right)
\sum_{j=1}^{k}{\deg_{L}(D_j)} < 0 ~.
$$
Thus, $\cE$ is not semistable with respect to $L$.
\end{proof}

Hence, by Corollary \ref{cor:number-rays}, we get:

\begin{cor}\label{cor:logtangent-unstability}
Let $p = \rk \Cl(X)$. If $1 + p \leq \card(\Delta) \leq n+p-1$, then for any
$L \in \Amp(X)$, the logarithmic tangent sheaf $\cT_{X}(- \log D)$ is unstable
with respect to $L$.
\end{cor}

\begin{rem}
By Corollary \ref{cor:filtration-logtangent}, if $\card(\Delta) = n + p$,
$\cT_{X}(- \log D)$ is semistable with respect to any polarizations.
\end{rem}

From now on, we will study the (semi)stability of $\cT_{X}(- \log D)$ only on the case
where $ 1 \leq \card(\Delta) \leq p = \rk \Cl(X)$ and $p \in \{1, 2\}$.

\section{Stability of equivariant logarithmic tangent sheaves}
\label{sec:stability-logtangent}

\subsection{Stability on weighted projective spaces}
Let $q_0, q_1, \ldots, q_n \in \N^{\ast}$ such that $$\gcd(q_0, \ldots, q_n) = 1.$$
We set $N = \Z^{n+1} / \Z \cdot (q_0, \ldots, q_n).$ The dual lattice of $N$ is
$$
M = \{(a_0, \ldots, a_n) \in \Z^{n+1} : a_0 \, q_0 + \ldots + a_n \, q_n= 0 \}.
$$
We denote by $\{ u_i : 0 \leq i \leq n \}$ the images in $N$ of the standard basis
vectors in $\Z^{n+1}$. So the relation
$q_0 \, u_0 + q_1 \, u_1 + \ldots + q_n \, u_n = 0$
holds in $N$. The toric variety $X$ associated to the simplicial fan 
$\Sigma = \{ \Cone(A) : A \subsetneq \{u_0, \ldots, u_n \} \}$
is the weighted projective space $\bP(q_0, q_1, \ldots, q_n)$.
We denote by $D_i$ the divisor of $X$ corresponding to the ray $\Cone(u_i)$.
For $i \in \{0, \ldots, n \}$, we set $\cE = \cT_{X}(- \log D_i)$ and
$A_i = \{0, \ldots, n \} \setminus \{i \}$.

\begin{prop}
Let $L \in \Amp(X)$. The sheaf $\cE$ is polystable with respect to $L$ if and only if
there is $q \in \N^{\ast}$ such that for all $j \in A_i$, $q_j = q$.
\end{prop}

\begin{proof}
We first show that $q_i D_j \sim_\lin q_j D_i$. Let $m = (a_0, \ldots, a_n) \in M$
defined by $a_i = q_j$, $a_j = -q_i$ and $a_k = 0$ if
$k \in A_i \setminus \{ j \}$. By Equation (\ref{eq:divisor-of-character}), we get
$\mathrm{div}(\chi^m) = q_j D_i - q_i D_j$. Hence, $q_i D_j \sim_\lin q_j D_i$.
Therefore, for any $L \in \Amp(X)$, $q_i \deg_L(D_j) = q_j \deg_L(D_i)$.

The assumptions of Proposition \ref{prop:logtangent-notsimple1} are verified. Hence,
$\cE = \bigoplus_{j \in A_i}{ \cE_{F_j} }$ where $F_j = \Span(u_j)$.
By Equation (\ref{eq:slope-logtangent2}), we get
$$
\mu_{L}(\cE_{F_j}) = \deg_{L}(D_j) = \dfrac{q_j}{q_i} \deg_{L}(D_i).
$$
If $\cE$ is polystable with respect to $L$, there is $r \in \Q$ such that for all
$j \in A_i$, $q_j = r \, q_i$. Hence, we have the existence of $q \in \N^{\ast}$
such that for all $j \in A_i$, $q_j = q$. For the converse, if for all $j \in A_i$,
we have $q_j = q$, then $\cE$ is polystable.
\end{proof}

According to Proposition \ref{prop:poly-vs-semistability}, we get :

\begin{cor}\label{cor:stability-logtgt-weighted}
For all $i \in \{0, \ldots, n \}$, $\sStab(\cT_{X}(- \log D_i)) \neq \nothing$ if and
only if there is $q \in \N^{\ast}$ such that for all $j \in A_i$, $q_j = q$. Moreover,
if for all $j \in A_i$, $q_j = q$, then
$$
\nothing = \Stab(\cT_{X}(- \log D_i)) \subsetneq \sStab(\cT_{X}(- \log D_i)) = \Amp(X).
$$
\end{cor}

\subsection{Condition of stability on toric varieties of Picard rank two}
\label{sec:stability-logtangent-toricpic2}
In this part, we adapt some results of \cite[Section 4]{HNS19} for the study of the
stability of $T_{X}(- \log D)$ when
$X = \bP \left( \cO_{\bP^s} \oplus \bigoplus_{i=1}^{r} \cO_{\bP^s}(a_i) \right)$
with $0 \leq a_1 \leq \ldots \leq a_r$.
We use notation of Section \ref{sec:toricpicard-two}.
The following lemma will be useful in the proof of Proposition
\ref{prop:stability-logtangent-toricpic2} which is the main result of this part.
Let $z \in \{0, \ldots, r-1 \}$ such that $a_{z} = 0$ and $a_{z+1} > 0$, we have:

\begin{lem}[{\cite[Lemma 4.2]{HNS19}}]\label{lem:position-rays-toricpic2}
Let $I' \subseteq \{0, \, 1 \ldots, r \}$ and $G = \Span(v_i : i \in I')$. The vector
$a_1 v_1 + \ldots + a_r v_r$ belongs to $G$ if and only if
\begin{enumerate}
\item[i.]
$\{z+1, \ldots, \, r \} \subseteq I' $ or
\item[ii.]
$\{0, \ldots, z \} \subseteq I'$,
$\card( \{z+1, \ldots, r \} \setminus I') \geq 1$ and $a_i = a_j$ for all
$i, j \in \{z+1, \ldots, r \} \setminus I'$.
\end{enumerate}
\end{lem}

Since passing to multiples of polarizations has no effect on stability, instead of
studying the stability of $T_X(-\log D)$ with respect to
$\pi^\ast \cO_{\bP^s}(\alpha) \otimes \cO_X(\beta)$, we will study the stability of
$T_{X}(- \log D)$ with respect to the $\Q$-divisor
$\pi^{\ast} \cO_{\bP^s}(\nu) \otimes \cO_{X}(1)$ where $\nu = \alpha / \beta$.
Let $P$ be the polytope corresponding to the $\Q$-polarized toric variety $(X, L)$ where
$L= \pi^{\ast} \cO_{\bP^s}(\nu) \otimes \cO_{X}(1)$ with $\nu \in \Q_{>0}$.

\begin{notation}
For all $i \in \{0, 1, \ldots, r \}$, we set $\rV_i = \vol(P^{v_i})$. As
for all $j \in \{1, \ldots, s \}$, $\vol(P^{w_j}) = \vol(P^{w_0})$, we set
$\rW= \vol(P^{w_0})$.
\end{notation}

Let $\Delta \subseteq \Sigma(1)$ and $D$ a reduced Weil divisor on $X$ given by
$D = \sum_{\rho \in \Delta}{D_{\rho}}$. We set
\begin{align*}
I_{\Sigma} & = \{\Cone(v_0), \ldots, \Cone(v_r) \} ~,
\\
J_{\Sigma} & = \{\Cone(w_0), \ldots, \Cone(w_s) \} ~,
\\
I & = \{ i \in \{0, 1, \ldots, r \} : \Cone(v_i) \in I_{\Sigma} \setminus
(I_{\Sigma} \cap \Delta ) \} \quad \text{and}
\\
J & = \{ j \in \{0, 1, \ldots, s \} : \Cone(w_j) \in J_{\Sigma} \setminus
(J_{\Sigma} \cap \Delta ) \} ~.
\end{align*}
To study the stability of $\cE = T_{X}(- \log D)$ with respect to $L$,
it suffices to compare $\mu_{L}(\cE)$ and $\mu_{L}(\cE_G)$ when
$G = \Span(v_i, w_j : i \in I', j \in J')$ with $I' \subseteq I$, $J' \subseteq J$ and
$1 \leq \dim G < (r+s)$.
By Proposition \ref{prop:degree-invariant-divisor}, (\ref{eq:slope-logtangent}) and
(\ref{eq:slope-logtangent2}), we get
$$
\mu_L(\cE) = \dfrac{1}{r+s} \left( \sum_{i \in I}{\rV_i} + \card(J) \cdot \rW \right)
$$
and
$$
\mu_{L}(\cE_G) = \dfrac{1}{\dim G} \left( \sum_{i \in I'}{\rV_i} +
\card(J') \cdot \rW \right) .
$$
Here is a version of \cite[Proposition 4.1]{HNS19} for logarithmic tangent bundle.

\begin{prop}\label{prop:stability-logtangent-toricpic2}
The logarithmic tangent bundle $\cE = T_{X}(- \log D)$ is stable (resp. semistable)
with respect to $L= \pi^{\ast} \cO_{\bP^s}(\nu) \otimes \cO_{X}(1)$ if and only if
$\mu_L(\cE)$ is greater than (resp. greater than or equal to) the maximum of
\begin{enumerate}[itemsep=2pt, topsep=2pt]
\item
$\rV_{i_0}$ where $i_0 = \min I$ if $I \neq \nothing$
\label{prop:stability-logtangent-item1};
\item
$\frac{1}{r'} \left( \sum_{i \in I}{\rV_i} \right)$,
if $r'= \dim \Span(v_i: i \in I) \neq 0$
\label{prop:stability-logtangent-item2};
\item
$\dfrac{\card(J) \cdot \rW}{s'}$, if $0 < s' = \dim \Span(w_j : j \in J) <r+s$
\label{prop:stability-logtangent-item3};
\item
$\frac{1}{s+k} \left( \sum_{i \in I'}{\rV_i} + (s+1)\rW \right)$, if
$\card(J') = s+1$, $k = \card(I') < r$ and $\{z+1, \ldots, r \} \subseteq I' \subseteq I$
\label{prop:stability-logtangent-item4};
\item
$\frac{1}{s+k} \left( \sum_{i \in I'}{\rV_i} + (s+1) \rW \right)$,
if $\card(J') = s+1$, $k = \card(I') < r$ and $I' \subseteq I$ such that the condition
{\it ii.} of Lemma \ref{lem:position-rays-toricpic2} is verified.
\label{prop:stability-logtangent-item5}
\end{enumerate}
\end{prop}

\begin{proof}
Let $G = \Span(v_i, \, w_j : \, i \in I', \, j \in J')$ where $I' \subseteq I$ and
$J' \subseteq J$.
In Proposition \ref{prop:stability-logtangent-toricpic2}, each point corresponds to a
value of $\mu_L(\cE_G)$ for some $G$. In particular,
(\ref{prop:stability-logtangent-item1}) corresponds to $G = \Span(v_{i_0})$,
(\ref{prop:stability-logtangent-item2}) corresponds to $G = \Span(v_i : i \in I)$
and (\ref{prop:stability-logtangent-item3}) corresponds to $G = \Span(w_j : j \in J)$.

If $\card(J') = 0$, then for $\nothing \subsetneq I' \subseteq I$, we have
$\dim G \leq r$ and
$$
\mu_{L}(\cE_G) = \dfrac{1}{\dim G} \sum_{i \in I'}{\rV_i}~;
$$
this number is less than or equal to the maximum of the numbers given in
(\ref{prop:stability-logtangent-item1}) and (\ref{prop:stability-logtangent-item2}).

If $\card(I') = 0$, then for $\nothing \subsetneq J' \subseteq J$ such that
$\dim G < r+s$, we have
$$
\mu_{L}(\cE_G) = \dfrac{\card(J') \cdot \rW}{\dim G} ~ ;
$$
this number is less than or equal to that given in
(\ref{prop:stability-logtangent-item3}).

If $\card(I') = r+1$, then $\dim G < r+s$ if and only if $s'= \card(J') < s$.
If $1 \leq  s' < s$, then
$$
\mu_{L}(\cE_G) = \dfrac{1}{r + s'} \left( \sum_{i \in I'}{\rV_i} + s' \rW \right)
\leq \max \left(\dfrac{1}{r} \sum_{i \in I'}{\rV_i} , \rW \right) ~ .
$$

If $1 \leq \card(I') \leq r$, $1 \leq \card(J') \leq s$ and $\dim G < r+s$, then
$\mu_{L}(\cE_G)$ is less than or equal to the maximum of numbers given in
(\ref{prop:stability-logtangent-item1}), (\ref{prop:stability-logtangent-item2}) and
(\ref{prop:stability-logtangent-item3}).

It remains to study the case where $\card(J') = s+1$ and $1 \leq \card(I') < r$
(because if $\card(I') \geq r$, then $\dim G = r+s$). We will treat it in two cases.
\\
{\it First case : $a_r = 0$.}
For all $i \in \{1, \ldots, r \}$, $\rV_i = \rV_0$. If $r'= \card(I')$ and
$1 \leq r' < r$, then
$$
\mu_{L}(\cE_G) = \dfrac{1}{r' + s} \left( \sum_{i \in I'}{\rV_i} + (s+1) \rW \right)
\leq \max \left(\rV_0 , \dfrac{(s+1) \rW}{s} \right) ~ .
$$
{\it Second case : $a_r >0$.}
We set $r' = \card(I')$. If $I'$ satisfies the first (resp. second) condition of Lemma
\ref{lem:position-rays-toricpic2}, then the value of $\mu_{L}(\cE_G)$ is given in the
point (\ref{prop:stability-logtangent-item4})
(resp. (\ref{prop:stability-logtangent-item5})).
If $I'$ doesn't satisfy the conditions of Lemma \ref{lem:position-rays-toricpic2},
then $\dim G = r' + (s+1)$. Moreover, if $r'+(s+1)< r+s$, then the number
$\mu_{L}(\cE_G)$ is less than or equal to the maximum of the numbers given in
(\ref{prop:stability-logtangent-item1}) and (\ref{prop:stability-logtangent-item3}).
\end{proof}

\begin{rem}
If $a_1 = \ldots = a_r = 0$, to check the stability of $\cE$ with respect to $L$, it is
enough to compare $\mu_{L}(\cE)$ with the numbers given by the points
\ref{prop:stability-logtangent-item1}, \ref{prop:stability-logtangent-item2} and
\ref{prop:stability-logtangent-item3} of Proposition
\ref{prop:stability-logtangent-toricpic2}. In that case, we have
\begin{equation}\label{eq:volfacet-product-projective}
\rW = \dbinom{s+r-1}{s-1} \nu^{s-1} \quad \text{and} \quad
\rV_i = \dbinom{s+r-1}{s} \nu^s .
\end{equation}
\end{rem}

If $(a_1, \ldots, a_r) \neq (0, \ldots, 0)$, the results below will help us to determine
if $\cE$ is unstable with respect to $L$ without having to check each point of
Proposition \ref{prop:stability-logtangent-toricpic2}.
Let $z \in \{0, 1, \ldots, r-1 \}$ such that $a_z = 0$ and $a_{z+1} > 0$ where $a_0 = 0$.
Let $k \in \{0, \ldots, s \}$. We set
\begin{align*}
\rV_{0k} & = \sum_{d_{z+1} + \ldots + d_r =s-k}{a_{z+1}^{d_{z+1}} \cdots a_{r}^{d_r} }
\quad \text{and}\\
\rW_k & = \sum_{d_{z+1} + \ldots + d_r =s-1-k}{ a_{z+1}^{d_{z+1}} \cdots a_{r}^{d_r} }
\end{align*}
where $\rW_s =0$. For $i \in \{z+1, \ldots, r \}$, we set
$$
\rV_{ik} = \sum_{\substack{d_{z+1} + \ldots + d_{i-1}\\ + d_{i+1} + \ldots + d_r=s-k}}{
a_{z+1}^{d_{z+1}} \cdots a_{i-1}^{d_{i-1}} a_{i+1}^{d_{i+1}} \cdots a_{r}^{d_r} }
$$
and for $i \in \{1, \ldots, z\}$, we set $\rV_{i k} = \rV_{0k}$.

\begin{rem}
If $r=1$, we set $\rV_{1s} = 1$ and for $k \in \{0, \ldots, s-1 \}$, $\rV_{1k} = 0$.
We have $\rW_{s-1} = 1$ and $\rV_{i s} = 1$ for any $i \in \{0, \ldots, r \}$.
\end{rem}

\begin{lem}\label{lem:relation-volfacet-toricpic2}
For all $i \in \{1, \ldots, r \}$, $\rV_0 = a_i \rW + \rV_i$.
\end{lem}

\begin{proof}
To show the lemma, it suffices to show that: for any $k \in \{0, \ldots, s-1 \}$,
$a_i \rW_k + \rV_{ik} = \rV_{0k}$.
If $i \in \{1, \ldots, z \}$, the equality is true because $a_i = 0$. We assume that
$i \in \{z+1, \ldots, r \}$, we have
\begin{align*}
\rV_{0k} & = \sum_{d_{z+1} + \ldots + d_r =s-k}{a_{z+1}^{d_{z+1}} \cdots a_{r}^{d_r} }
\\ & =
\sum_{\substack{d_{z+1}+ \ldots + d_r =s-k \\ d_i = 0}}{
a_{z+1}^{d_{z+1}} \cdots a_{r}^{d_r} } +
\sum_{\substack{d_{z+1}+ \ldots + d_r =s-k \\ d_i \geq 1}}{
a_{z+1}^{d_{z+1}} \cdots a_{r}^{d_r} }
\end{align*}
The first term of the second line corresponds to the number $\rV_{ik}$ and the
second to $a_i \rW_k$ (it suffices to replace $d_i$ by $d_{i}'+1$).
Hence, $\rV_{0k} = \rV_{ik} + a_i \rW_k$.
\end{proof}

\begin{lem}\label{lem:stability-logtangent-toricpic2}
Let $(a_1, \ldots, a_r) \neq (0, \ldots, 0)$.
\begin{enumerate}
\item If $a_r \geq 2$, then $s \rV_0 - (s+1) \rW \geq s \rV_r$.
\item If $ r \geq 2$ and $i \in \{1, \ldots, r-1 \}$ with $a_i < a_r$, then
$\rV_i - \rW \geq \rV_r$.
\end{enumerate}
\end{lem}

\begin{proof}
If $a_r \geq 2$, then
$\left( s - \dfrac{s+1}{a_r} \right) = \dfrac{a_r s - (s+1)}{a_r} \geq
\dfrac{2s -(s+1)}{a_r} \geq 0$ because $s \geq 1$. Hence,
\begin{align*}
s \rV_0 - (s+1)\rW & = s \rV_0 - \dfrac{s+1}{a_r}(\rV_0 - \rV_r)
\\ & =
\left( s - \dfrac{s+1}{a_r} \right) \rV_0 + \dfrac{s+1}{a_r} \rV_r
\\ & \geq
\left( s - \dfrac{s+1}{a_r} \right) \rV_r + \dfrac{s+1}{a_r} \rV_r = s \rV_r ~. 
\end{align*}
As $\rV_0 = a_i \rW + \rV_i = a_r \rW + \rV_r$, we get $\rV_i = (a_r - a_i) \rW + \rV_r$.
If $a_r > a_i$, then $a_r - a_i \geq 1$; therefore $\rV_i \geq \rW + \rV_r$.
\end{proof}

\subsection{Stability of logarithmic tangent bundles on a product of projective spaces}
\label{sec:stab-product-projective}
We assume that $a_1 = \ldots = a_r = 0$. We have
$X \cong \bP^{s} \times \bP^{r}$. We denote by $\pi_1 : X \rightarrow \bP^s$ and
$\pi_2 : X \rightarrow \bP^r$ the projection maps.
If $i \in \{0, \ldots, r \}$ and $j \in \{0, \ldots, s \}$, by Proposition
\ref{prop:logtangent-notsimple1} the vector bundle $T_{X}(- \log(D_{v_i} + D_{w_j}))$
is decomposable. Reasoning as in the proof of Proposition \ref{prop:logtangent-notsimple1},
it is easy to show that:

\begin{lem}\label{lem:logtangent-notsimple-prodproj}
Let $i, i' \in \{0, \ldots, r \}$ and $j, j' \in \{0, \ldots, s \}$ such that $i \neq i'$
and $j \neq j'$. Then
\begin{enumerate}
\item
$T_{X}(- \log(D_{v_i}+ D_{v_{i'}}) ) \cong
T_{\bP^s} \oplus T_{\bP^r}(- \log(\pi_2(D_{v_i}) + \pi_2(D_{v_{i'}}) \, ))$
\item
$T_{X}(- \log(D_{w_j}+ D_{w_{j'}}) ) \cong
T_{\bP^s}(- \log(\pi_1(D_{w_j}) + \pi_1(D_{w_{j'}}) \, )) \oplus T_{\bP^r}$
\item
$\cE = T_{X}(- \log D_{v_i})$ satisfies
$$
\cE \cong T_{\bP^s} \oplus \left( \bigoplus_{k=0, \, k \neq i}^{r}{ \cE_{F_k}} \right)
$$
where $F_k = \Span(v_k)$.
\item
$\cE = T_{X}(- \log D_{w_j})$ satisfies
$$
\cE \cong \left( \bigoplus_{k=0, \, k \neq j}^{s}{ \cE_{G_k}} \right) \oplus T_{\bP^r}
$$
where $G_k = \Span(w_k)$.
\end{enumerate} 
\end{lem}

Let $D= \sum_{\rho \in \Delta} D_\rho$ with $\Delta \subseteq \Sigma(1)$. As for any
$\Delta$ such that $\card(\Delta) \in \{1, 2 \}$, the vector bundle $\cE = T_X(-\log D)$
is decomposable, we deduce that
$$
\Stab(\cE) = \nothing ~.
$$
In Table \ref{tab:stab-logtgt-prodprojective}, we give the values of $\nu$ for which
$\cE$ is semistable with respect to $\pi^{\ast} \cO_{\bP^s}(\nu) \otimes \cO_{X}(1)$.
We recall that by Equation (\ref{eq:volfacet-product-projective}),
$\rV = \dfrac{r \nu}{s} \rW$.

\begin{table}
\setcellgapes{3pt}
\begin{tabular}{|C{4.5cm}|C{2.5cm}|c|}
\hline
Divisor $D$ & $\sStab(\cE)$ & References \\ \hline
$D_{v_i} ~$, $0 \leq i \leq r$ & $\nu = \dfrac{s+1}{r}$ &
Proposition \ref{prop:stab-prodprojective-logtgt1} \\ \hline
$D_{w_j}~$, $0 \leq j \leq s$ & $\nu = \dfrac{s}{r+1}$ &
Proposition \ref{prop:stab-prodprojective-logtgt1} \\ \hline
$D_{v_j} + D_{w_j}$ & $\nu = \dfrac{s}{r}$ &
Proposition \ref{prop:stab-prodprojective-logtgt3} \\ \hline
$D_{v_i} + D_{v_j} ~$, $0 \leq i < j \leq r$ & $\nothing$ &
Proposition \ref{prop:stab-prodprojective-logtgt2} \\ \hline
$D_{w_i} + D_{w_j}~$, $0 \leq i < j \leq s$ & $\nothing$ &
Proposition \ref{prop:stab-prodprojective-logtgt2} \\ \hline
\end{tabular}
\caption{Stability of $T_{X}(- \log D)$ when $a_1 = \ldots = a_r =0$}
\label{tab:stab-logtgt-prodprojective}
\end{table}

\begin{prop}\label{prop:stab-prodprojective-logtgt1}
Let $i \in \{0, 1, \ldots, r \}$ and $j \in \{0, 1, \ldots, s \}$, then
\begin{enumerate}
\item
$T_X(- \log D_{v_i})$ is polystable with respect to
$\pi^{\ast} \cO_{\bP^s}(\nu) \otimes \cO_{X}(1)$ if and only if $\nu = \dfrac{s+1}{r}$;
\item
$T_X(- \log D_{w_j})$ is polystable with respect to
$\pi^{\ast} \cO_{\bP^s}(\nu) \otimes \cO_{X}(1)$ if and only if $\nu = \dfrac{s}{r+1}$.
\end{enumerate}
\end{prop}

\begin{proof}
We start with $\cE = T_{X}(- \log D_{v_i} )$. We have
$$
\mu_L(\cE) = \dfrac{1}{r+s}(r \rV + (s+1)\rW) = \dfrac{(r^2 \nu + s^2 + s)\rW}{s(r+s)} =
\dfrac{(r^2 \nu + s^2 + s)\rV}{r \nu (r + s)} ~.
$$
By Proposition \ref{prop:stability-logtangent-toricpic2}, to have the semistability, it
is enough to compare $\mu_L(\cE)$ with
$$\max \left( \dfrac{(s+1) \rW}{s} , \rV \right).$$
If $\mu_L(\cE) \geq \rV$, then $\dfrac{r^2 \nu + s^2 + s}{r \nu (r + s)} \geq 1$, i.e
$(r^2 \nu + s^2 + s) \geq (r^2 \nu + rs \nu)$; hence, $\nu \leq \dfrac{s+1}{r}$.
\\
If $\mu_L(\cE) \geq \dfrac{(s+1) \rW}{s}$, then
$\dfrac{r^2 \nu + s^2 + s}{s(r+s)} \geq \dfrac{s+1}{s}$, i.e $\nu \geq \dfrac{s+1}{r}$.
Therefore, $T_{X}(- \log D_{v_i})$ is semistable with respect to
$\pi^{\ast} \cO_{\bP^s}(\nu) \otimes \cO_{X}(1)$ if and only if $\nu = \dfrac{s+1}{r}$.

If we regard the case where $\cE = T_{X}(- \log D_{w_j})$, it is enough to compare
$\mu_{L}(\cE)$ with $\max \left( \dfrac{(r+1) \rV}{r} , \rW \right)$. A similar
computation gives the result.
\end{proof}

\begin{prop}\label{prop:stab-prodprojective-logtgt2}
Let $i, i' \in \{0, \ldots, r \}$ and $j, j' \in \{0, \ldots, s \}$ such that $i \neq i'$
and $j \neq j'$. For any $ L \in \Amp(X)$, the logarithmic tangent bundles
$T_{X}(- \log(D_{v_i} + D_{v_{i'}}) )$ and $T_{X}(- \log (D_{w_j} + D_{w_{j'}}) )$
are unstable with respect to $L$.
\end{prop}

\begin{proof}
Let $\cE = T_{X}(- \log (D_{v_i} + D_{v_{i'}}) \,$. We have
\begin{align*}
\mu_L(\cE) & = \dfrac{(r-1)\rV + (s+1)\rW}{r+s} \\
& = \dfrac{(r(r-1) \nu + s^2 + s) \rW}{s(r+s)} \\
& = \dfrac{(r(r-1) \nu + s^2 + s) \rV}{r \nu (r + s)}.
\end{align*}
To check the semistability, it is enough to compare $\mu_L(\cE)$ with
$$\max \left( \dfrac{(s+1) \rW}{s} , \rV \right) .$$
If $r=1$, then $\mu_{L}(\cE) = \rW$. Hence, $T_{X}(- \log(D_{v_i} + D_{v_{i'}}) )$ is not
semistable with respect to $L$.
We now consider the case $r \geq 2$.
\begin{itemize}
\item
If $\mu_L(\cE) \geq \rV$, then $\dfrac{r(r-1) \nu + s^2 + s}{r \nu (r + s)} \geq 1$,
i.e $\nu \leq \dfrac{s}{r}$.
\item
If $\mu_L(\cE) \geq \dfrac{(s+1)\rW}{s}$, then
$\dfrac{r(r-1) \nu + s^2 + s}{s(r+s)} \geq \dfrac{s+1}{s}$, i.e
$\nu \geq \dfrac{s+1}{r-1} > \dfrac{s}{r}$.
\end{itemize}
As $\nu$ cannot satisfy this two conditions, we deduce that
$T_{X}(- \log(D_{v_i} + D_{v_{i'}}))$ is not semistable with respect to $L$.
\end{proof}

\begin{prop}\label{prop:stab-prodprojective-logtgt3}
Let $i \in \{0, \ldots, r \}$, $j \in \{0, \ldots, s \}$ and
$D = D_{v_i} + D_{w_j}$. Then $T_X(- \log D)$ is polystable with respect to
$\pi^{\ast} \cO_{\bP^s}(\nu) \otimes \cO_{X}(1) $ if and only if $\nu = \dfrac{s}{r}$.
\end{prop}

\begin{proof}
We have
$$
\mu_L(\cE) = \dfrac{1}{r+s}(r \rV + s \rW) = \dfrac{(r^2 \nu + s^2) \rW}{s(r+s)} =
\dfrac{(r^2 \nu + s^2) \rV}{r \nu (r + s)} ~.
$$
To check the semi-stability, it is enough to compare $\mu_L(\cE)$ with
$\max(\rV, \rW)$.
\begin{itemize}
\item
If $\mu_L(E) \geq \rV$, then $\dfrac{r^2 \nu + s^2 }{r \nu (r + s)} \geq 1$,
i.e $\nu \leq \dfrac{s}{r}$.
\item
If $\mu_L(E) \geq \rW$, then $\dfrac{r^2 \nu + s^2}{s(r+s)} \geq 1$,
i.e $\nu \geq \dfrac{s}{r}$.
\end{itemize}
Hence, $T_X(- \log D)$ is semistable with respect to
$\pi^{\ast} \cO_{\bP^s}(\nu) \otimes \cO_{X}(1)$ if and only if $\nu = \dfrac{s}{r}$.
\end{proof}

\begin{rem}\label{rem:stab-prodprojective}
According to (\ref{eq:divisor-toricpicard-two}) and (\ref{eq:canonical-toricpicard-two}),
when $a_1 = \ldots = a_r = 0$, we get
$D_{v_i} \sim_\lin D_{v_0}$, $D_{w_j} \sim_\lin D_{w_0}$ and
$-K_{X} \sim_\lin (s+1) D_{w_0} + (r+1) D_{v_0}.$
By the above study, we see that: if $\sStab( T_X(- \log D)) \neq \nothing$, then
$T_{X}(- \log D)$ is semistable with respect to $L$ if and only if
$L \cong \cO_{X}( - \alpha \, (K_X +D))$ with $\alpha \in \Q_{>0}$.
\end{rem}

\section{Stability on smooth toric varieties of Picard rank two}
\label{sec:stab-logtangent-toricpic2}

In this section, we study the stability of $T_{X}(- \log D)$ when
$$X = \bP \left( \cO_{\bP^s} \oplus \bigoplus_{i=1}^{r} \cO_{\bP^s}(a_i) \right)$$ with
$a_r \geq 1$. Let $\Delta \subseteq \Sigma(1)$ and $D=\sum_{\rho \in \Delta}{D_\rho}$.
By Corollary \ref{cor:logtangent-unstability}, we will only study the case where
$\card(\Delta) \in \{1, 2 \}$. The case $\card(\Delta) = 0$ was treated by
Hering-Nill-S\"uss in \cite{HNS19}.
We recall that $a_0 = 0$. We have a version of Lemma
\ref{lem:logtangent-notsimple-prodproj} when $a_r \geq 1$.

\begin{lem}\label{lem:logtangent-notsimple-toricpic2}
We assume that $a_r \geq 1$.
\begin{enumerate}
\item
If $i \in \{0, \ldots, r \}$ and $j \in \{0, \ldots, s \}$, then
$T_{X}(- \log(D_{v_i} + D_{w_j}))$ is decomposable and
$$
\cE = \left( \bigoplus_{l=0, \, l \neq j}^{s}{ \cE_{G_l}} \right) \oplus
\left( \bigoplus_{k=0, \, k \neq i }^{r}{ \cE_{F_k}} \right)
$$
where $G_l = \Span(w_l)$ and $F_k = \Span(v_k)$.
\item
If $D= D_{w_i}+D_{w_j}$ for $0 \leq i < j \leq s$, then the sheaf $\cE= T_{X}(- \log D)$
is decomposable and
$$
\cE = \left( \bigoplus_{k=0, \, k \neq i}^{s}{ \cE_{G_k}} \right) \oplus \cE_F
$$
where $G_k = \Span(w_k)$ and $F = \Span(v_0, \ldots, v_r)$.
\label{lem:logtangent-notsimple-item2}
\item
If $D= D_{v_i}+D_{v_j}$ for $0 \leq i < j \leq r$, then the sheaf $\cE= T_{X}(- \log D)$
is decomposable. If $a_i < a_j$, then
$$
\cE = \left( \bigoplus_{l=0}^{s}{ \cE_{G_l}} \right) \oplus
\left( \bigoplus_{k=0, \, k \notin \{i, j\} }^{r}{ \cE_{F_k}} \right)
$$
where $G_l = \Span(w_l)$ and $F_k = \Span(v_k)$. If $a_i = a_j$, then
$$
\cE = \cE_{G} \oplus \cE_{F}
$$
where
$G = \Span(w_l, v_k : l \in \{0, \ldots, s \},
k \in \{0, \ldots, r \} \setminus \{i, j \} )$ and $F = \Span(v_j)$.
\end{enumerate}
\end{lem}

If $D \in \{ D_{v_i} : 0 \leq i \leq r \}$, we will not search to know if
$\cE = T_{X}(- \log D)$ is decomposable because it depends on the numbers $r$, $s$ and
$a_1, \ldots, a_r$. In particular, if we assume that $r=2$, $s = 1$ and
$(a_1, a_2) = (0, 1)$, then $\cE = T_{X}(- \log D_{v_1})$ is decomposable with
$\cE = \cE_{F} \oplus \cE_G$ where $F = \Span(v_2, w_1)$ and $G = \Span(v_0)$ while
$\cF = T_{X}(- \log D_{v_2})$ is not decomposable.

For $D= D_{w_i}$ with $0 \leq i \leq s$, the vector bundle $\cE= T_{X}(-\log D)$ is
decomposable and its decomposition is identical to that given in the point
\ref{lem:logtangent-notsimple-item2} of Lemma \ref{lem:logtangent-notsimple-toricpic2}.
According to Lemma \ref{lem:logtangent-notsimple-toricpic2}, if
$D= \sum_{\rho \in \Delta} D_\rho$ with $\card(\Delta)=2$, then $\cE = T_{X}(- \log D)$
is not stable with respect to any polarizations.

Let $L = \pi^{\ast} \cO_{\bP^s}(\nu) \otimes \cO_{X}(1)$ be an element of
$\Amp(X) \subseteq N^{1}(X) \otimes_{\Z} \R$.
We recall that the numbers $\rV_0, \ldots, \rV_r$ defined on Section
\ref{sec:toricpicard-two} are polynomials of $\nu$ of degree $s$ and $\rW$ is
a polynomial of degree $s-1$.
If $\cE = T_{X}(- \log D)$, the number $\mu_{L}(\cE)$ is a polynomial of degree at
most $s$. Let $\rP_0$, $\rP_1$, $\rP_2$ and $\rQ$ be the polynomials of $\nu$ defined by
$$
\rP_0 = \mu_{L}(\cE) - \rV_0 ~,~ \rP_1 = \mu_{L}(\cE) - \rV_1 ~ ,~
\rP_2 = \mu_L(\cE) - \rV_2 ~\text{and}~ \rQ = \mu_{L}(\cE) - \rW ~.
$$
Under certain conditions on $a_i, r$ and $s$, these polynomials ($\rP_0, \rP_1, \rP_2$
and $\rQ$) have respectively one or no positive root. If the positive root exists, we
denote by
\begin{itemize}
\item
$\nu_i$ the unique positive root of $\rP_i$ where $i \in \{0, 1, 2 \}$
\item
$\nu_3$ the unique positive root of $\rQ$.
\end{itemize}
In Tables \ref{tab:stab-logtgt-toricpic2-1}, \ref{tab:stab-logtgt-toricpic2-2} and
\ref{tab:stab-logtgt-toricpic2-3}, we give the values of $\nu$ for which
$\cE = T_{X}(- \log D)$ is (semi)stable with respect to
$L = \pi^{\ast} \cO_{\bP^s}(\nu) \otimes \cO_{X}(1)$.

\begin{table}
\setcellgapes{3pt}
\begin{tabular}{|L{2.1cm}|C{2.1cm}|C{2cm}|c|c|}
\hline
Divisor $D$ & Condition on $r$ and $a_i$ & Condition on $s$ & $\Stab(\cE)$
& $\sStab(\cE)$ \\ \hline
\makecell[l]{$D_{w_j}$\\ $0 \leq j \leq s$\\ Prop. \ref{prop:logtgt-toricpic-Dw}}
& \makecell{$r \geq 1$ and\\ $a_r \geq 1$} & $s \geq 1$ & $\nothing$ & $\nothing$ \\ \hline
\makecell[l]{$D_{v_i}$\\ $1 \leq i \leq r-1$\\ Prop. \ref{prop:logtgt-toricpic-Dv}}
& \makecell{$r \geq 2$ and \\$a_r \geq 1$} & $s \geq 1$ & $\nothing$ & $\nothing$ \\ \hline
$D_{v_r}$ & \makecell{$r \geq 1, a_{r} = 1$\\ and $a_{r-1} = 0$} & $s \geq 1$
& $0< \nu < \nu_0 $ & $0 < \nu \leq \nu_0$ \\ \cline{2-5}
Theorem \ref{theo:logtgt-toricpic-Dvr} &
\makecell{$r \geq 1$ and \\($a_r \geq 2$ or\\ $a_{r-1} \neq 0$)} & $s \geq 1$
& $\nothing$ & $\nothing$ \\ \hline
& $r=1$ & $s \geq 1$ & $0 < \nu < \nu_1$ & $0 < \nu \leq \nu_1$ \\ \cline{2-5}
$D_{v_0}$ & \makecell{$r \geq 2$ and\\ $a_1 < a_r$} & $s \geq 1$ & $\nothing$ & $\nothing$
\\ \cline{2-5}
Theorem \ref{theo:logtgt-toricpic-Dv0} &
$r \geq 2 $ and & $a \geq \frac{s+1}{r-1} $ & $\nothing$ & $\nothing$ \\ \cline{3-5}
Lemma \ref{lem:logtgt-toricpic-Dv0} &
$a_1 = a_r = a$ & $\frac{s}{r} \leq a < \frac{s+1}{r-1}$
& $0 < \nu < \nu_1 $ & $0< \nu \leq \nu_1$ \\ \cline{3-5}
Theorem \ref{theo:logtgt-toricpic-Dv01} &
& $a \, r < s$ & $\nu_3 < \nu < \nu_1$ & $\nu_3 \leq \nu \leq \nu_1$ \\ \hline
\end{tabular}
\caption{Stability of $T_{X}(- \log D)$ when $a_r \geq 1$}
\label{tab:stab-logtgt-toricpic2-1}
\end{table}

\begin{table}
\setcellgapes{3pt}
\begin{tabular}{|L{3.5cm}|C{2.7cm}|c|c|}
\hline
Divisor $D$ & Condition on $r$ and $a_i$ & Condition on $s$ & $\sStab(\cE)$ \\ \hline
\makecell[l]{$D_{w_i} + D_{w_j}$\\ $0 \leq i < j \leq s$\\
Proposition \ref{prop:logtgt-toricpic-Dw}} & $r \geq 1$ and $a_r \geq 1$ & $s \geq 1$
& $\nothing$ \\ \hline
\makecell[l]{$D_{v_i} + D_{v_j}$\\ $1 \leq i < j \leq r$\\
Corollary \ref{cor:logtgt-toricpic-Dv}} &
$r \geq 2$ and $a_r \geq 1$ & $s \geq 1$ & $\nothing$ \\ \hline
\makecell[l]{$D_{v_i} + D_{w_j}~$, $j \geq 0$\\ and $1 \leq i \leq r-1$\\
Proposition \ref{prop:logtgt-toricpic-Dv}}
& $r \geq 2$ and $a_r \geq 1$ & $s \geq 1$ & $\nothing$ \\ \hline
\makecell[l]{$D_{v_r} + D_{w_j}~$, $j \geq 0$ \\
Corollary \ref{cor:logtgt-toricpic-Dvr}} &
$r \geq 1$ and $a_r \geq 1$ & $s \geq 1$ & $\nothing$ \\ \hline
$D_{v_0} + D_{w_j}~$, $0 \leq j \leq s$ & $r=1$ & $s \geq 1$ & $\nu=\nu_3$ \\ \cline{2-4}
Theorem \ref{theo:logtgt-toricpic-Dv0} &
\makecell{$r \geq 2$ and\\ $a_1 < a_r$} & $s \geq 1$ & $\nothing$ \\ \cline{2-4}
Lemma \ref{lem:logtgt-toricpic-Dv0} &
$r \geq 2$ and & $s \leq a(r-1)$ & $\nothing$ \\ \cline{3-4}
Proposition \ref{prop:logtgt-toricpic-Dv02} &
$a_1 = a_r = a$ & $s > a(r-1)$ & $\nu = \nu_3$ \\ \hline
$D_{v_0} + D_{v_i}~$, $2 \leq i \leq r$ & \makecell{$r \geq 2$ and\\ $a_1 < a_r$} &
$s \geq 1$ & $\nothing$ \\ \cline{2-4}
Lemma \ref{lem:logtgt-toricpic-Dv0} &
$r \geq 2$ and & $s \leq a(r-1)$ & $\nothing$ \\ \cline{3-4}
Theorem \ref{prop:logtgt-toricpic-Dv02} &
$a_1 = a_r = a$ & $s > a(r-1)$ & $\nu = \nu_3$ \\ \hline
\end{tabular}
\caption{Stability of $T_{X}(- \log D)$ when $a_r \geq 1$}
\label{tab:stab-logtgt-toricpic2-2}
\end{table}

\begin{table}
\setcellgapes{3pt}
\begin{tabular}{|L{2.8cm}|c|c|c|}
\hline
Divisor $D$ & Condition on $r$ and $a_i$ & Condition on $s$ & $\sStab(\cE)$ \\ \hline
& $r=1$ & $s \geq 1$ & $\nu >0$ \\ \cline{2-4}
$D_{v_0} + D_{v_1}$ & $r \geq 2$ and $0 = a_1 < a_r$ & $s \geq 1$ & $\nothing$
\\ \cline{2-4}
& $r \geq 2$ and & $s \leq a(r-1)$ & $\nothing$ \\ \cline{3-4}
Theorem \ref{theo:logtgt-toricpic-Dv0} &
$a_1 = a_r = a$ & $s > a(r-1)$ & $\nu = \nu_3$ \\ \cline{2-4}
Proposition \ref{prop:logtgt-toricpic-Dv0Dv1} &
$r = 2$ and & $s \leq \delta_2 $ & $\nothing$ \\ \cline{3-4}
Proposition \ref{prop:logtgt-toricpic-Dv02} &
$0< a_1 < a_2 $ & $s > \delta_2$ & $\nu = \nu_3$ \\ \cline{2-4}
Proposition \ref{prop:logtgt-toricpic-Dv0Dv1r2} &
$r \geq 3$ and $a_2 < a_r$ & $s \geq 1$ & $\nothing$ \\ \cline{2-4}
Proposition \ref{prop:logtgt-toricpic-Dv0Dv1r3} &
$r \geq 3$ and & $s \leq \delta_r$ & $\nothing$ \\ \cline{3-4}
& $0< a_1< a_2 = \ldots = a_r$ & $s > \delta_r $ & $\nu = \nu_3$ \\ \hline
\end{tabular}
\caption{Stability of $T_{X}(- \log(D_{v_0} + D_{v_1}))$ when $a_r \geq 1$}
\label{tab:stab-logtgt-toricpic2-3}
\end{table}

\subsection{Case of divisors coming from the base}

\begin{prop}\label{prop:logtgt-toricpic-Dw}
Let $(a_1, \ldots, a_r) \neq (0, \ldots, 0)$.
Let $i, j \in \{0, \ldots,s \}$ distinct, $\cE = T_{X}(- \log D_{w_i})$ and
$\cF = T_{X}(- \log( D_{w_i} + D_{w_j}))$.
For any $L \in \Amp(X)$, the vector bundles
$\cE$ and $\cF$ are not semistable with respect to $L$.
\end{prop}

\begin{proof}
Let $\, L = \pi^{\ast} \cO_{\bP^s}(\nu) \otimes \cO_{X}(1)$, we have
$$
\mu_L(\cF) = \dfrac{(s-1) \rW + (\rV_0 + \ldots + \rV_{r})}{r+s} <
\dfrac{s \rW + (\rV_0 + \ldots + \rV_{r})}{r+s} = \mu_L(\cE).
$$
By Lemma \ref{lem:logtangent-notsimple-toricpic2} and Proposition
\ref{prop:stability-logtangent-toricpic2}, to check the stability of $\cE$ (resp. $\cF$)
with respect to $L$, it is enough to compare $\mu_{L}(\cE)$ (resp. $\mu_{L}(\cF)$) with
$$
\max \left( \rV_0, \dfrac{\rV_0 + \rV_1 + \ldots + \rV_r}{r} \right)~.
$$
By Lemma \ref{lem:relation-volfacet-toricpic2}, we have $\rV_0 = a_r \rW + \rV_r$.
As $a_r \geq 1$, we get $\rV_0 \geq \rW + \rV_r$, i.e $\rV_0 - \rW \geq \rV_r$. Thus,
\begin{align*}
(r+s) \left( \rV_0 - \mu_{L}(\cE) \right) & =
s(\rV_0 - \rW) + \left( r \rV_0 - (\rV_0 + \ldots + \rV_{r-1}) \right) - \rV_r
\\ & \geq
s(\rV_0 - \rW) - \rV_r \quad \text{because} ~ \rV_i \leq \rV_0
\\ & \geq (s-1) \rV_r ~.
\end{align*}
If $s \geq 2$, then $\rV_0 - \mu_{L}(\cE) > 0$ and $\rV_0 - \mu_{L}(\cF) >0$.
Thus, $\cE$ and $\cF$ are not semistable with respect to $L$.
We now assume that $s = 1$. Using the expressions of $\rV_i$ and $\rW$ given in Section
\ref{sec:toricpicard-two}, we have
$$
\rW = 1 ~, \quad \rV_0 = (a_1 + \ldots + a_r) + r \nu \quad \text{and} \quad
\rV_i = \rV_0 - a_i \quad \text{for} ~ i \in \{1, \ldots, r \} ~. 
$$
As
$$
\mu_{L}(\cE) = \dfrac{1 + (r+1)\rV_0 - (a_1 + \ldots + a_r)}{r+1}
$$
and
$$
\dfrac{\rV_0 + \ldots + \rV_r}{r} = \dfrac{(r+1)\rV_0 - (a_1 + \ldots + a_r)}{r} ~,
$$
we get
$$
\dfrac{\rV_0 + \ldots + \rV_r}{r} - \mu_{L}(\cE) = \dfrac{
(r+1) \rV_0 - (a_1 + \ldots + a_r) - r}{r(r+1)} ~ .
$$
As $(a_1 + \ldots + a_r - 1) \geq 0$ and $ \nu >0$, we have
\begin{align*}
(r+1) \rV_0 - (a_1 + \ldots + a_r) - r = &(r+1)(a_1 + \ldots + a_r) + (r+1)r \nu
\\ & - (a_1 + \ldots + a_r) - r
\\ = & r (a_1 + \ldots + a_r - 1) + (r+1)r \nu >0 
\end{align*}
Thus, $\cE$ and $\cF$ are not semistable with respect to $L$.
\end{proof}

\subsection{Sum of divisors coming from the base and the bundle: first part}
We first study the stability of $T_{X}(- \log D)$ when $r \geq 2$ and
\begin{align*}
D \in & \{ D_{v_i} : 1 \leq i \leq r-1 \} \cup
\{D_{v_i} + D_{w_j} : 1 \leq i \leq r-1 ~ \text{and} ~ 0 \leq j \leq s \}
\\ & \cup \{D_{v_i} + D_{v_j} : 1 \leq i < j \leq r \} ~. 
\end{align*}

\begin{prop}\label{prop:logtgt-toricpic-Dv}
Let $r \geq 2$, $(a_1, \ldots, \, a_r) \neq (0, \ldots, 0)$,
$i \in \{1, \ldots, r-1 \} \,$ and $j \in \{0, \ldots, s \}$. For any $L \in \Amp(X)$,
the logarithmic tangent bundles $T_{X}(- \log D_{v_i})$ and
$T_{X}(- \log(D_{v_i} + D_{w_j}))$ are not semistable with respect to $L$.
\end{prop}

\begin{proof}
We set $\cE = T_{X}(- \log D_{v_i})$ and $\cF = T_{X}(- \log(D_{v_i} + D_{w_j}))$.
Let $L \in \Amp(X)$, we have
$$
\mu_L(\cE) = \dfrac{(s+1) \rW + (\rV_0 + \ldots + \rV_{i-1} + \rV_{i+1} + \ldots +
\rV_r)}{r+s}
$$
and $\mu_L(\cF) < \mu_{L}(\cE)$. By Lemma \ref{lem:stability-logtangent-toricpic2},
\begin{align*}
(r+s)( \rV_0 - \mu_L(\cE) ) = & (s+1)(\rV_0 - \rW) - \rV_r + (r-1) \rV_0 \\ &
- (\rV_0 + \ldots + \rV_{i-1} + \rV_{i+1} + \ldots + \rV_{r-1})
\\ \geq & (s+1)(\rV_0 - \rW) - \rV_r
\\ \geq & (s+1) \rV_r - \rV_r = s \rV_r
\end{align*}
Hence, by Proposition \ref{prop:stability-logtangent-toricpic2}, we deduce that $\cE$
and $\cF$ are not semistable with respect to $L$.
\end{proof}

\begin{cor}\label{cor:logtgt-toricpic-Dv}
Let $r \geq 2$ and $(a_1, \ldots, a_r) \neq (0, \ldots, 0)$. Let $L \in \Amp(X)$,
for any $i,j \in \{1, \ldots, r \}$ with $i \neq j$, the logarithmic tangent bundle
$T_{X}(- \log( D_{v_i} + D_{v_j}))$ is not semistable with respect to $L$.
\end{cor}

\begin{proof}
If we set $\cG = T_{X}(- \log(D_{v_i} + D_{v_j}))$, by using the proof of Proposition
\ref{prop:logtgt-toricpic-Dv}, we have $\mu_{L}(\cG) < \mu_{L}(\cE) < \rV_0$. Thus,
$\cG$ is not semistable with respect to $L$.
\end{proof}

We now study the stability of $T_{X}(- \log D)$ when
$ D \in \{ D_{v_r} \} \cup \{D_{v_r} + D_{w_j} : 0 \leq j \leq s \} ~.$

\begin{theorem}\label{theo:logtgt-toricpic-Dvr}
Let $r \geq 1$ and $a_r \geq 1$. We have $\Stab( T_{X}(- \log D_{v_r})) \neq \nothing$
if and only if $\sStab( T_{X}(- \log D_{v_r})) \neq \nothing$ if and only if $a_r = 1$
and $a_{r-1} = 0$. If $a_r = 1$ and $a_{r-1} = 0$, then the logarithmic tangent bundle
$T_{X}(- \log D_{v_r})$ is stable (resp. semi-stable) with respect to
$\pi^{\ast} \cO_{\bP^s}(\nu) \otimes \cO_{X}(1)$ if and only if $0 < \nu < \nu_0$
(resp. $0 < \nu \leq \nu_0$) where $\nu_0$ is the positive root of
$$
\rP_0(x) = \sum_{k = 0}^{s-1} \dbinom{s+r-1}{k} x^k - s \dbinom{s+r-1}{s} x^s ~.
$$
\end{theorem}

\begin{proof}
Let $\cE = T_{X}(- \log D_{v_r})$ and
$L = \pi^{\ast} \cO_{\bP^s}(\nu) \otimes \cO_{X}(1)$. We have
$$
(r+s) \mu_{L}(\cE) = (s+1) \rW + \rV_0 + \rV_1 + \ldots + \rV_{r-1} ~.
$$
If $a_r \geq 2$, using the first point of Lemma \ref{lem:stability-logtangent-toricpic2}
and the fact that $\rV_i \leq \rV_0$, we get:
$$
(r+s)[\rV_0 - \mu_{L}(\cE)] = (s \rV_0 - (s+1)\rW) + r \rV_0 - (\rV_0 + \ldots + 
\rV_{r-1}) \geq s \rV_r.
$$
By Proposition \ref{prop:stability-logtangent-toricpic2}, $T_{X}(- \log D_{v_r})$ is not
semistable with respect to $L$.

We assume that $r \geq 2$ and $a_{r-1} = a_r = 1$. As $\rV_{r-1} = \rV_r$, we have
\begin{align*}
(r+s)[\rV_0 - \mu_{L}(\cE)] & = (s+1)[\rV_0 - \rW] - \rV_{r-1}
\\ & \quad + [ (r-1) \rV_0 -
(\rV_0 + \ldots + \rV_{r-2})]
\\
& \geq (s+1) \rV_r - \rV_{r-1} \quad \text{because} ~ \rV_0 - \rW \geq \rV_r
\\ & \geq s \rV_r
\end{align*}
By Proposition \ref{prop:stability-logtangent-toricpic2}, $T_{X}(- \log D_{v_r})$ is not
semistable with respect to $L$.

Let $r \geq 1$. We now assume that $a_{r-1} = 0$ and $a_r = 1$.
By using the expressions of Section \ref{sec:toricpicard-two}, we have
$\rV_0 = \ldots = \rV_{r-1} = \rV$ where
$$
\rV = \sum_{k=0}^{s}{ \dbinom{s+r-1}{k} \nu^k } \quad \text{and} \quad
\rW = \sum_{k=0}^{s-1}{ \dbinom{s+r-1}{k} \nu^k } ~.
$$
The points \ref{prop:stability-logtangent-item4} and \ref{prop:stability-logtangent-item5}
of Proposition \ref{prop:stability-logtangent-toricpic2} are not verified in this case.
To check the stability of $\cE$ it is enough to compare
$$
\mu_{L}(\cE) = \dfrac{r \rV + (s+1)\rW}{r+s}
$$
with $\max(\rV, \rW)$. We have
$ (r+s)(\mu_{L}(\cE) - \rW) = r \rV - (r-1) \rW > 0$ because $\rW < \rV$ and
\begin{align*}
(r+s)(\mu_{L}(\cE) - \rV) =& (s+1)\rW - s \rV
\\ = &
\sum_{k=0}^{s-1}{ \dbinom{s+r-1}{k} \nu^k } - s \dbinom{s+r-1}{s} \nu^s =
\rP_0(\nu) ~~.
\end{align*}
By the sign rule of Descartes, the polynomial
$\rP_0$ have a unique positive root $\nu_0$. If $\nu >0$, then $\rP_0(\nu) > 0$
(resp. $\rP_0(\nu) \geq 0$) if and only if $\nu < \nu_0$ (resp. $\nu \leq \nu_0$).
Thus, $T_{X}(- \log D_{v_r})$ is stable (resp. semistable) with respect to
$\pi^{\ast} \cO_{\bP^s}(\nu) \otimes \cO_{X}(1)$ if and only if $0 < \nu < \nu_0$
(resp. $0 < \nu \leq \nu_0 $).
\end{proof}

\begin{cor}\label{cor:logtgt-toricpic-Dvr}
We assume that $r \geq 1$ and $(a_1, \ldots, a_r) \neq (0, \ldots, 0)$.
Let $j \in \{0, \ldots, s \}$ and $D = D_{v_r} + D_{w_j}$. For any $L \in \Amp(X)$, the
logarithmic tangent bundle $T_{X}(- \log D)$ is not semistable with respect to $L$.
\end{cor}

\begin{proof}
If $\cE = T_{X}(- \log D)$, we have $\rV_0 > \mu_{L}(\cE)$. By Proposition
\ref{prop:stability-logtangent-toricpic2}, $T_{X}(- \log D)$ is not semistable with
respect to $L$.
\end{proof}

\subsection{Sum of divisors coming from the base and the bundle: second part}
In this part we study the stability of the logarithmic tangent bundle
$T_{X}(- \log D)$ when $r \geq 2$ and
$$
D \in \{ D_{v_0} \} \cup \{ D_{v_0} + D_{w_j} : 0 \leq j \leq s \} \cup
\{D_{v_0} + D_{v_i} : 2 \leq i \leq r \} ~.
$$
The last case $D = D_{v_0} + D_{v_1}$ will be studied in Section
\ref{sec:logtgt-toricpic-Dv0Dv1}.

\begin{lem}\label{lem:logtgt-toricpic-Dv0}
Let $r \geq 2$, $(a_1, \ldots, a_r) \neq (0, \ldots, 0)$ such that $a_1 < a_r$,
$i \in \{2, \ldots, r \}$ and $j \in \{0, \ldots, s \}$. We set
$\cE = T_{X}(- \log D_{v_0})$, $\cF = T_{X}(- \log(D_{v_0} + D_{v_i}))$ and
$\cG = T_{X}( - \log(D_{v_0} + D_{w_j}))$. For any $L \in \Amp(X)$, the vector bundles
$\cE$, $\cF$ and $\cG$ are not semistable with respect to $L$.
\end{lem}

\begin{proof}
We have $\mu_L(\cE) > \mu_{L}(\cF)$ and $\mu_{L}(\cE) > \mu_{L}(\cG)$. We will show that
$\rV_1 > \mu_{L}(\cE)$. By Lemma \ref{lem:stability-logtangent-toricpic2}, we have
$\rV_1 - \rW \geq \rV_r$. Therefore
\begin{align*}
(r+s)(\rV_1 - \mu_L(\cE) ) & = (r+s)\rV_1 - (\rV_1 + \ldots + \rV_r) - (s+1)\rW
\\ & = (s+1)(\rV_1 - \rW) - \rV_r
\\ & \quad + (r-1)\rV_1 - (\rV_1 + \ldots + \rV_{r-1})
\\ & \geq (s+1) (\rV_1 - \rW) - \rV_r
\\ & \geq s \rV_r
\end{align*}
By Proposition \ref{prop:stability-logtangent-toricpic2}, $\cE$, $\cF$ and $\cG$ are not
semistable with respect to $L$.
\end{proof}

Let $a \in \N^{\ast}$. We now study what happen in Lemma \ref{lem:logtgt-toricpic-Dv0}
when $a_1 = \ldots = a_r = a$. We first consider the case $r=1$.

\begin{theorem}\label{theo:logtgt-toricpic-Dv0}
We assume that $X = \bP \left( \cO_{\bP^s} \oplus \cO_{\bP^s}(a) \right)$.
Let $\rP_1$ and $\rQ$ be the polynomials defined by
$$
\rP_1(x) = (s+1) \sum_{k=0}^{s-1}{ \dbinom{s}{k}a^{s-k-1} x^k} - s \, x^s
~\text{and}~
\rQ(x) = x^s - \sum_{k=0}^{s-1}{ \dbinom{s}{k}a^{s-k-1} x^k} .
$$
Then:
\begin{enumerate}[leftmargin=*]
\item
$T_{X}(- \log D_{v_0})$ is stable (resp. semistable) with respect to
$\pi^{\ast} \cO_{\bP^s}(\nu) \otimes \cO_{X}(1)$ if and only if $0 < \nu < \nu_1$
(resp. $0 < \nu \leq \nu_1$) where $\nu_1$ is the unique positive root of $\rP_1$.
\item
If $j \in \{0, \ldots, s \}$, then $T_{X}(- \log(D_{v_0} + D_{w_j}))$ is semistable with
respect to $\pi^{\ast} \cO_{\bP^s}(\nu) \otimes \cO_{X}(1)$ if and only if $\nu = \nu_3$
where $\nu_3$ is the unique positive root of $\rQ$.
\item
$\nothing = \Stab( T_{X}(- \log(D_{v_0} + D_{v_1})) ) \subsetneq
\sStab( T_{X}(- \log(D_{v_0} + D_{v_1})) ) = \Amp(X).$
\end{enumerate}
\end{theorem}

\begin{proof}
By the sign rule of Descartes, $\rP_1$ and $\rQ$ have
respectively one positive root. Let $L= \pi^{\ast} \cO_{\bP^s}(\nu) \otimes \cO_{X}(1)$,
by using the expressions of Section \ref{sec:toricpicard-two}, we have :
$$
\rV_1 = \nu^s \quad \text{and} \quad
\rW = \sum_{k=0}^{s-1}{\dbinom{s}{k} a^{s-k-1} \nu^k } ~~.
$$
By Proposition \ref{prop:stability-logtangent-toricpic2}, to check the stability of
$\cE = T_{X}(- \log D_{v_0})$, it is enough to compare
$$
\mu_L(\cE) = \dfrac{\rV_1 + (1+s) \rW}{1+s}
$$
with $\max(\rV_1, \rW)$. We have $\mu_{L}(\cE) > \rW$ and
$(1+s)(\mu_{L}(\cE) - \rV_1) = \rP_1(\nu).$ Thus, $\cE$ is stable (resp. semi-stable)
with respect to $L$ if and only if $0 < \nu < \nu_1$ (resp. $0 < \nu \leq \nu_1$).

Let $\cF = T_{X}(- \log(D_{v_0} + D_{w_j}))$. By Proposition
\ref{prop:stability-logtangent-toricpic2}, it is enough to compare
$$
\mu_L(\cF) = \dfrac{\rV_1 + s \rW}{1+s}
$$
with $\max(\rV_1, \rW)$. As $(1+s)(\mu_{L}(\cF) - \rW) = \rQ(\nu)$ and
$(1+s)(\mu_{L}(\cF) - \rV_1) = -s \rQ(\nu)$,
we deduce that $\cF$ is semistable with respect to $L$ if and only if $\nu = \nu_3$.

Let $\cG = T_{X}(- \log(D_{v_0} + D_{v_1}))$. We have $\mu_{L}(\cG) = \rW$.
By Proposition \ref{prop:stability-logtangent-toricpic2}, $\cG$ is semistable with
respect to $L$.
\end{proof}

We now consider the case $r \geq 2$ and $a_1 = \ldots = a_r = a$ with $a \in \N^{\ast}$.

\begin{lem}\label{lem:cardinal-set}
We have
$$
\card \{(\alpha_1, \ldots, \alpha_p) \in \N^p : \alpha_1 + \ldots + \alpha_p= m \} =
\dbinom{m + p -1}{m} .
$$
\end{lem}

We recall that $\rV_{1s} = 1$. By Lemma \ref{lem:cardinal-set}, for all
$k \in \{0, \ldots, s-1 \}$,
$$
\rW_k = \sum_{d_1 + \ldots + d_r = s-k-1}{a_{1}^{d_1} \cdots a_{r}^{d_r}} =
\dbinom{s-k+r-2}{s-k-1} a^{s-k-1}
$$
and
$$
\rV_{1k} = \sum_{d_2 + \ldots + d_r = s-k}{a_{2}^{d_2} \cdots a_{r}^{d_r}} =
\dbinom{s-k+r-2}{s-k} a^{s-k} ~~.
$$
Using the equality
$\dbinom{n}{p-1} = \dfrac{p}{n-p+1} \dbinom{n}{p}$, for any $k \in \{0, \ldots, s-1 \}$,
$$
\rW_k = \dfrac{s-k}{r-1} \dbinom{s-k+r-2}{s-k} a^{s-k-1} = \dfrac{s-k}{a(r-1)} \rV_{1k} ~.
$$

\begin{theorem}\label{theo:logtgt-toricpic-Dv01}
Let $r \geq 2$ and
$X = \bP( \cO_{\bP^s} \oplus \bigoplus_{i=1}^{r} \cO_{\bP^s}(a_i) )$ with
$a_1 = \ldots = a_r = a$ where $a \in \N^{\ast}$.
We set $\cE = T_{X}(- \log D_{v_0})$. Let $\rP_1$ and $\rQ$ be the polynomials defined by:
\begin{align*}
\rP_1(x) &= \sum_{k=0}^{s-1}{ \left[ \left(\dfrac{(s-k)(s+1)}{a(r-1)}-s \right)
\dbinom{s+r-1}{k} \rV_{1k} \right] x^k} - s \dbinom{s+r-1}{s} x^s
\\
\rQ(x) &= \sum_{k=0}^{s-1}{ \left[ \left( r - \dfrac{s-k}{a} \right)
\dbinom{s+r-1}{k} \rV_{1k} \right] x^k} + r \dbinom{s+r-1}{s} x^s
\end{align*}
Then:
\begin{enumerate}[itemsep=2pt, topsep=2pt, leftmargin=*]
\item
If $a < \dfrac{s}{r}$, then $\cE$ is stable (resp. semistable) with respect to
$\pi^{\ast} \cO_{\bP^s}(\nu) \otimes \cO_{X}(1)$ if and only if $\nu_3 < \nu < \nu_1$
(resp. $\nu_3 \leq \nu \leq \nu_1$) where $\nu_1$ and $\nu_3$ are respectively the
positive roots of $\rP_1$ and $\rQ$.
\item
If $\dfrac{s}{r} \leq a < \dfrac{s+1}{r-1}$, then $\cE$ is stable (resp. semistable) with
respect to $\pi^{\ast} \cO_{\bP^s}(\nu) \otimes \cO_{X}(1)$ if and only if
$0 < \nu < \nu_1$ (resp. $0 < \nu \leq \nu_1$) where $\nu_1$ is the positive root of
$\rP_1$.
\item
If $a \geq \dfrac{s+1}{r-1}$, then for any $L \in \Amp(X)$, $\cE$ is not semistable with
respect to $L$.
\end{enumerate}
\end{theorem}

\begin{proof}
We first explain the condition which ensure the existence of positive roots on $\rP_1$
and $\rQ$.
We write
$$
\rP_1(x) = \sum_{k=0}^{s}{\alpha_k \, x^k} \quad \text{and} \quad
\rQ(x) = \sum_{k=0}^{s}{\beta_k \, x^k} ~.
$$
For $k \in \{0, \ldots, s-1 \}$, $\alpha_k >0$ if and only if
$k < \left( 1 - \dfrac{a(r-1)}{s+1} \right) s$. Therefore,
\begin{itemize}
\item
If $~\dfrac{a(r-1)}{s+1} \geq 1$, then for any $x \geq 0$, $\rP_1(x) < 0$.
\item
If $~\dfrac{a(r-1)}{s+1} <1$, then $\rP_1$ has only one positive root $\nu_1$.
\end{itemize}
For $k \in \{0, \ldots, s-1 \}$, $\beta_k < 0$ if and only if $k < s - r a$.
Therefore,
\begin{itemize}
\item
If $r a \geq s$, then for any $x \geq 0$, $\rQ(x) > 0$.
\item
If $r a < s$, then $\rQ$ has only one positive root $\nu_3$.
\end{itemize}
We now show that : If $ a < \dfrac{s}{r}$, then $\nu_3 < \nu_1$. As
\begin{align*}
\dfrac{\rP_1(x)}{-s} - \dfrac{\rQ(x)}{r} & =
\sum_{k=0}^{s-1}{ \left[ \left(-\dfrac{(s-k)(s+1)}{a \, s (r-1)} +
\dfrac{s-k}{r \, a} \right) \dbinom{s+r-1}{k} \rV_{1k} \right] x^k }
\\ & =
\dfrac{-(r+s)}{a \, s \, r (r-1)} \sum_{k=0}^{s-1}{
(s-k) \dbinom{s+r-1}{k} \rV_{1k} \, x^k } = \rP(x)
\end{align*}
and $\dfrac{\rP_1(\nu_3)}{-s} - \dfrac{\rQ(\nu_3)}{r} = \rP(\nu_3) < 0$, we deduce that
$\rP_1(\nu_3) > 0$.
By using the fact that, for $x \geq 0$, $\rP_1(x) > 0$ if and only if $0 \leq x < \nu_1$,
we deduce that $\nu_3 < \nu_1$.

We can now study the stability of $\cE$. As $a_1 = \ldots = a_r$, we have
$\rV_1 = \ldots = \rV_r$. Therefore
$$
\mu_{L}(\cE) = \dfrac{r \rV_1 + (s+1) \rW }{r+s} ~.
$$
By Proposition \ref{prop:stability-logtangent-toricpic2}, to check the stability of $\cE$,
it is enough to compare $\mu_L(\cE)$ with $\max(\rV_1, \rW)$. We have
\begin{align*}
(r+s) ( \mu_L(\cE) - \rV_1) & = -s \rV_1 + (s+1) \rW
%\\ & =
%-s \sum_{k=0}^{s}{ \dbinom{s+r-1}{k} \rV_{1k} \, \nu^k} + (s+1)
%\sum_{k=0}^{s-1}{ \dbinom{s+r-1}{k} \rW_k \, \nu^k}
%\\ &
= \rP_1(\nu)
\end{align*}
and
\begin{align*}
(r+s) ( \mu_L(\cE) - \rW) & = r \rV_1 - (r-1) \rW
%\\ & =
%r \sum_{k=0}^{s}{ \dbinom{s+r-1}{k} \rV_{1k} \, \nu^k} - (r-1)
%\sum_{k=0}^{s-1}{ \dbinom{s+r-1}{k} \rW_k \, \nu^k}
%\\ &
= \rQ(\nu)
\end{align*}
Therefore,
\begin{enumerate}[itemsep=4pt, topsep=4pt]
\item[i.]
If $a \geq \dfrac{s+1}{r-1}$, then for any $\nu >0$, $\rP_1(\nu) < 0$.
\item[ii.]
If $a < \dfrac{s+1}{r-1}$, then $\rP_1(\nu) > 0$ (resp. $\rP_1(\nu) \geq 0$)
if and only if $0 < \nu < \nu_1$ (resp. $0 < \nu \leq \nu_1$).
\item[iii.]
If $a \geq \dfrac{s}{r}$, then for any $\nu > 0$, $\rQ(\nu) >0 \,$.
\item[iv.]
If $a < \dfrac{s}{r}$, then $\rQ(\nu) > 0$ (resp. $\rQ(\nu) \geq 0$)
if and only if $\, \nu > \nu_3$ (resp. $\, \nu \geq \nu_3$).
\end{enumerate}
The point {\it i.} shows the third point of the theorem. By using the points {\it ii.}
and {\it iv.}, we get the first point of theorem. Finally, the points {\it ii.} and
{\it iii.} give the second point of the theorem.
\end{proof}

\begin{prop}\label{prop:logtgt-toricpic-Dv02}
Let $r \geq 2$ and
$X = \bP( \cO_{\bP^s} \oplus \bigoplus_{i=1}^{r} \cO_{\bP^s}(a_i) )$ with
$a_1 = \ldots = a_r = a$ where $a \in \N^{\ast}$.
Let $i \in \{1, \ldots, r \}$ and $j \in \{0, \ldots, s \}$. We set
$\cF_j = T_{X}( - \log(D_{v_0} + D_{w_j}))$, $\cG_i = T_{X}(- \log(D_{v_0} + D_{v_i}))$
and
$$
\rQ(x) = \sum_{k=0}^{s-1}{ \left[ \left(1 - \dfrac{s-k}{a(r-1)} \right)
\dbinom{s+r-1}{k} \rV_{1k} \right] x^k} + \dbinom{s+r-1}{s} x^s ~.
$$
\begin{enumerate}[leftmargin=*]
\item
If $a \geq \dfrac{s}{r-1}$, then for any $L \in \Amp(X)$, $\cF_j$ and $\cG_i$ are not
semistable with respect to $L$.
\item
If $a < \dfrac{s}{r-1}$, then $\cF_j $ and $\cG_i$ are semistable with respect to
$\pi^{\ast} \cO_{\bP^s}(\nu) \otimes \cO_{X}(1)$ if and only if $\nu = \nu_3$
where $\nu_3$ is the unique root of $\rQ$.
\end{enumerate}
\end{prop}

\begin{proof}
We first study the polynomial $\rQ$. We write $\rQ(x) = \sum_{k=0}^{s}{\alpha_k \, x^k}$.
For $k \in \{0, \ldots, s-1 \}$, $\alpha_k >0$ if and only if $k < s - a(r-1)$.
\begin{itemize}
\item
If $~ a \geq \dfrac{s}{r-1}$, then for any $x \geq 0$, $\rQ(x) > 0$.
\item
If $~ a < \dfrac{s}{r-1}$, then $\rQ$ has a unique positive root $\nu_3$.
\end{itemize}
As $a_1 = \ldots = a_r$, we have $\rV_1 = \ldots = \rV_r$. Thus,
$$
\mu_{L}(\cF_j) = \dfrac{r \rV_1 + s \rW}{r+s} \quad \text{and} \quad
\mu_{L}(\cG_i) = \dfrac{(r-1) \rV_1 + (s+1) \rW}{r+s} ~.
$$
By Proposition \ref{prop:stability-logtangent-toricpic2}, to check the stability of
$\cF_j$ (resp. $\cG_i$), it is enough to compare $\mu_{L}(\cF_j)$ (resp. $\mu_L(\cG_i)$)
with $\max(\rV_1, \rW)$. We have
$$
\left\lbrace
\begin{array}{ll}
(r+s)(\mu_L(\cF_j) - \rV_1) = s(\rW - \rV_1) = - s \rQ(\nu) \\
(r+s)(\mu_L(\cF_j) - \rW) = r(\rV_1 - \rW) = r \rQ(\nu)
\end{array}
\right.
$$
and
$$
\left\lbrace
\begin{array}{ll}
(r+s)(\mu_L(\cG_i) - \rV_1) = (s+1)(\rW - \rV_1) = - (s+1) \rQ(\nu) \\
(r+s)(\mu_L(\cG_i) - \rW) = (r-1)(\rV_1 - \rW) = (r-1) \rQ(\nu)
\end{array}
\right.
$$

If $a \geq \dfrac{s}{r-1}$, then for any $\nu >0$, $Q(\nu) >0$; thus,
$\mu_L(\cF_j) < \rV_1$ and $\mu_L(\cG_i) < \rV_1$. Hence, for any $\nu >0$, $\cF_j$ and
$\cG_i$ are not semistable with respect to $L$.

If $a < \dfrac{s}{r-1}$, then by the above equalities, $\cF_j$ and $\cG_i$ are semistable
with respect to $\pi^{\ast} \cO_{\bP^s}(\nu) \otimes \cO_{X}(1)$ if and only if
$\nu = \nu_3$ where $\nu_3$ is the positive root of $\rQ$.
\end{proof}

\subsection{Sum of divisors coming from the bundle}
\label{sec:logtgt-toricpic-Dv0Dv1}

In this part, we assume that $\cE = T_{X}(- \log(D_{v_0} + D_{v_1}))$. We study the
stability of $\cE$ when $r \geq 2$ and $a_1 < a_r$. The stability of $\cE$ when $r=1$ was
treated in Theorem \ref{theo:logtgt-toricpic-Dv0}. When $r \geq 2$, in Proposition
\ref{prop:logtgt-toricpic-Dv02}, we studied the stability of $\cE$ when
$a_1 = \ldots = a_r$.

\begin{prop}\label{prop:logtgt-toricpic-Dv0Dv1}
Let $(a_1, \ldots, a_r) \neq (0, \ldots, 0)$ and $$\cE= T_{X}(- \log(D_{v_0} + D_{v_1})).$$
\begin{enumerate}
\item
If $a_1 = 0$, then for any $L \in \Amp(X)$, $\cE$ is not semistable with respect to $L$.
\item
If $r \geq 3$ and $a_2 < a_r$, then for any $L \in \Amp(X)$, $\cE$ is not semistable with
respect to $L$.
\end{enumerate}
\end{prop}

\begin{proof}
We have
$$
\mu_{L}(\cE) = \dfrac{(s+1) \rW + \rV_2 + \ldots + \rV_r}{r+s} .
$$
{\it First point.}
As $\card \{2, \ldots, r \} = r-1$, by using the point
\ref{prop:stability-logtangent-item4} of Proposition
\ref{prop:stability-logtangent-toricpic2} with $I' = \{2, \ldots, r \}$, we get
$$
\dfrac{1}{r+s-1}\left( \sum_{i \in I'}{\rV_i} + (s+1) \rW \right) = \dfrac{1}{r+s-1}
\left( \rV_2 + \ldots + \rV_r + (s+1) \rW \right) ~.
$$
Thus, $\cE$ is not semistable with respect to $L$.
\\
{\it Second point.}
By Lemma \ref{lem:stability-logtangent-toricpic2}, we have $\rV_2 - \rW \geq \rV_r$.
Therefore,
\begin{align*}
(r+s)(\rV_2 - \mu_L(\cE) & = (r+s)\rV_2 - (\rV_2 + \ldots + \rV_r) - (s+1)\rW
\\ & =
(s+1)(\rV_2 - \rW) + ((r-1)\rV_2 - (\rV_2 + \ldots + \rV_{r}))
\\ & \geq
(s+1) (\rV_2 - \rW)
\\ & \geq (s+1) \rV_r
\end{align*}
Hence, by Proposition \ref{prop:stability-logtangent-toricpic2}, $\cE$ is not semistable
with respect to $L$.
\end{proof}

We now assume that $0 < a_1 < a_2 = \ldots = a_r$. By Proposition
\ref{prop:stability-logtangent-toricpic2}, to check the stability of $\cE$, it is enough
to compare $\mu_{L}(\cE)$ with $\max(\rV_2, \rW)$. We have
$$
\mu_{L}(\cE) = \dfrac{(r-1)\rV_2 + (s+1)\rW}{r+s}
$$
and
\begin{equation}\label{eq:logtgt-toricpic-Dv0Dv1}
\left\lbrace
\begin{array}{l}
(r+s)[\mu_{L}(\cE) - \rV_2] = (s+1)(\rW - \rV_2) \\
(r+s)[\mu_{L}(\cE) - \rW] = -(r-1)(\rW - \rV_2)
\end{array}
\right.
\end{equation}
The vector bundle $\cE$ is semistable with respect to
$L = \pi^{\ast} \cO_{\bP^s}(\nu_3) \otimes \cO_{X}(1)$ if and only if $\nu_3$ is a positive
root of the polynomial $\rQ(\nu) = \rW - \rV_2$ ($\rW$ and $\rV_2$ depend on $\nu$).
We first consider the case $r=2$.

\begin{prop}\label{prop:logtgt-toricpic-Dv0Dv1r2}
Let $r = 2$ and $0 < a_1 < a_2$. We define
$\delta = \dfrac{\ln(1+a_2 - a_1)}{\ln(a_2) - \ln(a_1)}$ and the polynomial $\rQ$ by
$$
\rQ(x) = \sum_{k=0}^{s-1}{ \left[ \dfrac{a_{1}^{s-k}}{a_2 - a_1} \left(
\left( \dfrac{a_2}{a_1} \right)^{s-k} - 1 -a_2 + a_1 \right) \dbinom{s+1}{k}
\right] x^k} - (s+1) x^{s}~.
$$
Then:
\begin{enumerate}
\item
If $s \leq \delta$, then $\sStab( \, T_{X}(- \log(D_{v_0} + D_{v_1})) \, ) = \nothing$;
\item
If $s > \delta$, then $T_{X}(- \log(D_{v_0} + D_{v_1}))$ is semistable with respect to
$\pi^{\ast} \cO_{\bP^s}(\nu) \otimes \cO_{X}(1)$ if and only if $\nu = \nu_3$ where
$\nu_3$ is the positive root of $\rQ$.
\end{enumerate}
\end{prop}

\begin{proof}
We have
\begin{align*}
\rV_2 &= \sum_{k=0}^{s}{ \dbinom{s+1}{k} a_1^{s-k} \, \nu^k}
\\
\rW &= \sum_{k=0}^{s-1}{ \dbinom{s+1}{k} \left(
\sum_{d_1 + d_2 = s-k-1}{a_{1}^{d_1} a_{2}^{d_2}} \right) \, \nu^k }
= \sum_{k=0}^{s-1}{ \dfrac{a_2^{s-k} - a_{1}^{s-k}}{a_2 - a_1} \dbinom{s+1}{k} \, \nu^k }
\end{align*}
hence,
\begin{align*}
\rW - \rV_2 = & \sum_{k=0}^{s-1}{ \left[ \dfrac{a_{1}^{s-k}}{a_2 - a_1} \left(
\left( \dfrac{a_2}{a_1} \right)^{s-k} - 1 -a_2 + a_1 \right) \dbinom{s+1}{k}
\right] \nu^k}
\\ & - (s+1) \nu^{s} = \rQ(\nu) ~.
\end{align*}
We write $\rQ(x) = \sum_{k=0}^{s}{\alpha_k \, x^k}$. The inequality
$\left( \dfrac{a_2}{a_1} \right)^{s-k} - 1 -a_2 + a_1 > 0$ gives
$$
k < s - \dfrac{\ln(1+a_2 - a_1)}{\ln(a_2) - \ln(a_1)} = s - \delta .
$$
Hence, by the Descartes rule, $\rQ$ has a unique positive root $\nu_3$ if and only if
$s > \delta$.
\end{proof}

We now consider the case where $r \geq 3$ and $a_1 < a_r$. Let $a, b \in \N^{\ast}$ such
that $a < b$. We assume that $a_1 = a$ and $a_2 = \ldots = a_r = b$.
By Lemma \ref{lem:cardinal-set}, for any $k \in \{0, \ldots, s-1 \}$, we have
\begin{align*}
\rW_k = \sum_{\substack{d_1 + \ldots + d_r \\ = s-k-1}}{a_{1}^{d_1} \cdots a_{r}^{d_r}}
& =
\sum_{j=0}^{s-k-1}{ a^{s-k-1-j} \left( \sum_{d_2 + \ldots + d_r = j}{b^j } \right) }
\\ & =
\sum_{j = 0}^{s-k-1}{ \dbinom{j+r-2}{j} b^j \, a^{s-k-1-j} }
\end{align*}
and
\begin{align*}
\rV_{2k}= \rV_{rk} = \sum_{\substack{d_1 + \ldots + d_{r-1} \\ = s-k}}{
a_{1}^{d_1} \cdots a_{r-1}^{d_{r-1}} }
& =
\sum_{j=0}^{s-k}{a^{s-k-j} \left( \sum_{d_2 + \ldots + d_{r-1} = j}{b^j} \right) }
\\ & =
\sum_{j = 0}^{s-k}{ \dbinom{j+r-3}{j} b^j \, a^{s-k-j} } ~.
\end{align*}
For $p \in \{1, \ldots, s \}$, we set $\alpha_p = \rW_{s-p} - \rV_{2,s-p}$. We have
$$
\alpha_p =
\sum_{j = 0}^{p-1}{ \dbinom{j+r-2}{j} b^j \, a^{p-1-j} } -
\sum_{j = 0}^{p}{ \dbinom{j+r-3}{j} b^j \, a^{p-j} } ~.
$$
Let $\rQ_s$ be the polynomial defined by
$$
\rQ_{s}(x) = \sum_{k=0}^{s-1}{ \dbinom{s+r-1}{k}{\alpha_{s-k} \, x^k}} -
\dbinom{s+r-1}{s} x^s ~.
$$
We have $\rW - \rV_2 = \rQ_{s}(\nu)$. We now search a condition on $s$ which
ensure the existence of positive root on $\rQ_s$. By using the identity 
$\dbinom{n}{p-1} = \dfrac{p}{n-p+1} \dbinom{n}{p}$, we have
\begin{align*}
\alpha_p & = \sum_{j=0}^{p-1}{ \dfrac{j+1}{r-2} \dbinom{j+r-2}{j+1} b^j \, a^{p-1-j}} -
\sum_{j=1}^{p}{ \dbinom{j+r-3}{j} b^{j} \, a^{p-j}} - a^p
\\ & =
\sum_{j=0}^{p-1}{\left[ \left( \dfrac{j+1}{r-2} - b \right) \dbinom{j+r-2}{j+1}
b^{j} \, a^{p-1-j} \right] } - a^p ~~.
\end{align*}
If $1 \leq p \leq b(r-2)$, then for all $j \in \{0, \ldots, p-1 \}$, we have
$$
\dfrac{j+1}{r-2} - b \leq \dfrac{p}{r-2} - b = \dfrac{p - b(r-2)}{r-2} \leq 0 ~;
$$
thus $\alpha_p < 0$. Hence, if $\alpha_p > 0$, then we must have $p > b(r-2)$.
If there is $p > b(r-2)$ such that $\alpha_p > 0$, then for any $q \geq p$, we have
$\alpha_q > 0$; this follows from these equalities.
\begin{align*}
\alpha_{p+1} = & \sum_{j=0}^{p-1}{\left[ \left( \dfrac{j+1}{r-2} - b \right)
\dbinom{j+r-2}{j+1} b^{j} \, a^{p-j} \right] } - a^{p+1}
\\ & +
\left( \dfrac{p+1}{r-2} - b \right) \dbinom{p+r-2}{p+1} b^{p}
\\ = &
a \, \alpha_p + \left( \dfrac{p+1}{r-2} - b \right) \dbinom{p+r-2}{p+1} b^{p}
\end{align*}
We denote by $\lfloor x \rfloor$ the floor of $x \in \R$.

\begin{lem}\label{lem:logtgt-toricpic-Dv0Dv1r3}
Let $m = b(r-2)$. There is a unique integer
$\delta_r \in \left[ m+1 \, ; \, \lfloor 3.2 m \rfloor + 1 \right]$ such that: if
$p \leq \delta_r$, then $\alpha_p \leq 0$ and if $p > \delta_r$, then $\alpha_{p} > 0$.
\end{lem}

Let $\delta_r$ be the integer given in Lemma \ref{lem:logtgt-toricpic-Dv0Dv1r3}.
If $s \leq \delta_r$, then all coefficients of $\rQ_s$ are negative; thus, for any
$x >0$, $\rQ_s(x) < 0$. If $s > \delta_r$, then by the Descartes rule, $\rQ_s$ has
a only one positive root $\nu_3$. We deduce :

\begin{prop}\label{prop:logtgt-toricpic-Dv0Dv1r3}
Let $\, r \geq 3 \,$ and $\, a, \, b \in \N^{\ast}$ such that $\, a< b \,$ and
$a_1 = a \, $, $a_2 = \ldots = a_r = b \,$.
\begin{enumerate}
\item
If $\, s \leq \delta_r \,$, then
$\, \sStab( \, T_{X}(- \log(D_{v_0} + D_{v_1})) \, ) = \nothing \,$;
\item
If $s > \delta_r$, then $T_{X}(- \log(D_{v_0} + D_{v_1}))$ is
semistable with respect to $\pi^{\ast} \cO_{\bP^s}(\nu) \otimes \cO_{X}(1)$
if and only if $\nu = \nu_3$.
\end{enumerate}
\end{prop}

We now give the proof of Lemma \ref{lem:logtgt-toricpic-Dv0Dv1r3}.

\begin{proof}
We have
$$
\alpha_p =
\sum_{j=0}^{m-1}{\left[ \left( \dfrac{j+1}{r-2} - b \right) \dbinom{j+r-2}{j+1}
b^{j} \, a^{p-1-j} \right] } - a^p + \beta_p
$$
where
\begin{align*}
\beta_p & = \sum_{l=m}^{p-1}{\left[ \left( \dfrac{l+1 - (r-2)b}{r-2} \right)
\dbinom{l+r-2}{l+1} b^{l} \, a^{p-1-l} \right] }
\\ & =
\sum_{l=0}^{p-1-m}{ \left[ \dfrac{l+1}{r-2} \dbinom{l+m+r-2}{l+m+1}
b^{l+m} \, a^{p-1-(l+m)} \right]}
\end{align*}
The goal of this proof is to find an integer $p$ such that $\alpha_p > 0$. We will
search an integer $p$ such that
\begin{equation}\label{eq:logtgt-toricpic-Dv0Dv1r3}
\beta_p \geq a^p + \sum_{j=0}^{m-1}{ b \dbinom{j+r-2}{j+1} b^{j} \, a^{p-1-j}} ~.
\end{equation}
We have
$\dbinom{l+m+r-2}{l+m+1} = \dbinom{l+m+r-2}{r-3} = \dfrac{r-2}{l+m+1}
\dbinom{l+m+r-2}{r-2}.$
From the equality
$$
\sum_{j=n}^{r}{ \dbinom{j}{n} } = \dbinom{r+1}{n+1} \quad \text{we have} \quad
\dbinom{r+1}{n+1} = 1 + \sum_{j=0}^{r-n-1}{ \dbinom{j+n+1}{n} }~;
$$
Hence,
\begin{align*}
\dbinom{l+m+r-2}{r-2} = 1 + \sum_{j=0}^{l+m-1}{\dbinom{j+r-2}{r-3}} & =
1 + \sum_{j=0}^{l+m-1}{\dbinom{j+r-2}{j+1}}
\\
& \geq 1 + \sum_{j=0}^{m-1}{ \dbinom{j+r-2}{j+1}} .
\end{align*}
Thus,
\begin{align*}
\beta_p & \geq \left( 1 + \sum_{j=0}^{m-1}{\dbinom{j+r-2}{j+1}} \right)
\sum_{l=0}^{p-1-m}{ \dfrac{l+1}{l+1+m} \, b^{l+m} \, a^{p-1-(l+m) } }
\\ & \geq
b^{m} \, a^{p-1-m}
\left( 1 + \sum_{j=0}^{m-1}{\dbinom{j+r-2}{j+1}} \right)
\sum_{l=0}^{p-1-m}{ \dfrac{l+1}{l+1+m} \left(\dfrac{b}{a}\right)^l } ~~.
\end{align*}
We have
\begin{align*}
\sum_{l=0}^{p-1-m}{\dfrac{l+1}{l+1+m}} & \geq
\sum_{l=0}^{p-1-m}{ \int_{l}^{l+1}{\dfrac{x \, dx}{x+m}} }
\\ & \geq
\int_{0}^{p-m}{ \dfrac{x}{x+m} dx}
\\ & \geq p-m - m \ln \left( \dfrac{p}{m} \right).
\end{align*}
If $k = \dfrac{p}{m} \geq 3.2$, then
$(k-1- \ln(k)) > 1$. If we set $p = \lfloor 3.2 \, m \rfloor + 1$, we get
$$
\sum_{l=0}^{p-1-m}{ \dfrac{l+1}{l+1+m} \left(\dfrac{b}{a}\right)^l } \geq
\sum_{l=0}^{p-1-m}{ \dfrac{l+1}{l+1+m} } \geq m \geq b ~.
$$
For $j \in \{0, \ldots, m-1 \}$, we have
$b^{m} \, a^{p-1-m} \geq b^j \, a^{p-1-j}$. If $p = \lfloor 3.2 \, m \rfloor + 1$,
we get
\begin{align*}
\beta_p & \geq b \, b^{m} \, a^{p-1-m}
\left( 1 + \sum_{j=0}^{m-1}{\dbinom{j+r-2}{j+1}} \right)
\\ & \geq
b^{m+1} \, a^{p-1-m} + \sum_{j=0}^{m-1}{ b \dbinom{j+r-2}{j+1} b^{j} \, a^{p-1-j}}
~;
\end{align*}
this proves the inequality (\ref{eq:logtgt-toricpic-Dv0Dv1r3}). Thus, $\alpha_{p} > 0$
for $p = \lfloor 3.2 \, m \rfloor + 1$. Hence, we deduce the existence of the integer
$\delta_r$ in the interval $[m+1 \, ; \, \lfloor 3.2 \, m \rfloor + 1]$.
\end{proof}

\section{Application on log smooth toric del Pezzo pairs}
\label{sec:logdpezzo-pairs}

The goal of this part is to study the stability of the equivariant logarithmic tangent
bundle $T_{X}(- \log D)$ with respect to $-(K_X + D)$ when the pair $(X, D)$ is log del
Pezzo. We assume that $N = M = \Z^2$ and the pairing
$\< \cdot , \cdot \> : M \times N \rightarrow \Z$ is the usual dot product.

\begin{example}\label{examp:Hirzebruch-surface}
Let $r \in \N$ and $\Sigma$ the fan of the Hirzebruch surface
$\F_r = \bP( \cO_{\bP^1} \oplus \cO_{\bP^1}(r) )$.
The rays of $\Sigma$ are the half lines generated by the vectors
$u_1 = e_1$, $u_2 = e_2$, $u_3 = -e_1 + r \, e_2$ and $u_0 = -e_2$. Hence,
$$
\Sigma = \{0 \} \cup \left\lbrace \Cone(u_i) : 0 \leq i \leq 3 \right\rbrace \cup
\left\lbrace \Cone(u_i, u_{i+1}) : 0 \leq i \leq 3 \right\rbrace
$$
where $u_4 = u_0$. For any $i \in \{0, \ldots, 3\}$, we denote by $D_i$ the divisor
corresponding to the ray $\Cone(u_i)$. By \cite[Proposition 6.4.4]{CLS}, we have
\begin{equation}\label{eq:intersection-smooth-dim2}
\left\lbrace
\begin{array}{ll}
D_i \cdot D_i = - \gamma_i & \\
D_k \cdot D_i = 1 & \text{if} ~ k \in \{i-1, i+1 \} \\
D_k \cdot D_i = 0 & \text{if} ~ k \notin \{i-1, i, i+1 \}
\end{array}
\right.
\end{equation}
where $\gamma_i = \det(u_{i-1}, u_{i+1})$. So, $\gamma_0= -r$, $\gamma_1= 0$,
$\gamma_2= r$ and $\gamma_3= 0$.
If $\pi: \F_r \rightarrow \bP^1$ is the projection map, then the invariant divisors
$D_1, D_3$ are the fibers of $\pi$ and the invariant divisors $D_0, D_2$ can be seen
as sections.
\end{example}

By using the classification of log Del Pezzo surfaces given by Maeda
\cite[\S 3.4]{Maeda-fano3} (see e.g. \cite{Nap-DPezzo} for a proof in a toric setting),
we get the following description of equivariant log Del Pezzo pairs.

\begin{prop}
Let $X$ be a smooth complete toric surface and $D$ a reduced torus-invariant divisor
on $X$. Then, the pair $(X, D)$ is log Del Pezzo if:
\begin{enumerate}[itemsep=2pt, topsep=2pt]
\item
$X= \bP^2$ and $D= D'$ where $D'$ is a line;
\item
$X=\bP^2$ and $D= D' + D''$ where $D'$ and $D''$ are two lines;
\item
$X= \F_r$ and $D= D'$ where $D'$ is a section with $(D')^2 = - r$;
\item
$X= \F_r$ and $D= D' + D''$ where $D'$ is a section with $(D')^2 = - r$ and
$D''$ is a fiber;
\item
$X= \F_1$ and $D= D'$ where $D'$ is a section such that $(D')^2 = 1$;
\item
$X= \F_0$ and $D= D''$ where $D''$ is a fiber.
\end{enumerate}
\end{prop}

\begin{rem}
If $D, D'$ are two invariant lines of $\bP^2$, then according to Corollary
\ref{cor:stability-logtgt-weighted} and Corollary \ref{cor:logtangent-unstability},
$T_{\bP^2}(- \log D)$ is polystable with respect to $-(K_{\bP^2} + D)$ and
$T_{\bP^2}(- \log(D + D'))$ is unstable with respect to $-(K_{\bP^2} + D + D')$.
\end{rem}

\begin{prop}\label{prop:stability-wrt-anticanonical-dim2}
Let $X = \F_r$ and $D_0, D_1, D_2, D_3$ the divisors defined in Example
\ref{examp:Hirzebruch-surface}. Then:
\begin{enumerate}[itemsep=2pt, topsep=2pt]
\item
If $r=0$ and $ D \in \{D_i : 0 \leq i \leq 3 \} \cup \{D_0 + D_1, D_0 + D_3 \} \cup
\{D_2 + D_1, D_2 + D_3 \}$, $T_{X}(-\log D)$ is polystable with respect to $-(K_X + D)$;
\item
If $r=1$, $T_{X}(- \log D_0)$ is stable with respect to $-(K_X + D_0)$;
\item
If $r \geq 1$ and $D \in \{D_2, D_2 + D_1, D_2 + D_3 \}$, $T_{X}(- \log D)$ is unstable
with respect to $-(K_X + D)$.
\end{enumerate}
\end{prop}

\begin{proof}
We first note that, the divisors $D_0, D_1, D_2, D_3$ of $\F_r$ defined in Example
\ref{examp:Hirzebruch-surface} are given in Section \ref{sec:toricpicard-two} by
$D_0 = D_{v_0}$, $D_1 = D_{w_1}$, $D_2 = D_{v_1}$ and $D_3 = D_{w_0}$ where $v_1 = e_2$
and $w_1 = e_1$. Thus, by Equation (\ref{eq:divisor-toricpicard-two}),
$$
D_1 \sim_\lin D_3 \quad \text{and} \quad D_2 \sim_\lin D_0 - r D_3 ~.
$$
If $\alpha D_3 + \beta D_0$ is an ample divisor of $\F_r$, then the number $\nu$ used in
the results of Sections \ref{sec:stab-product-projective} and
\ref{sec:stab-logtangent-toricpic2} is defined by $\nu= \dfrac{\alpha}{\beta}$.
Using Remark \ref{rem:stab-prodprojective}
and Propositions \ref{prop:stab-prodprojective-logtgt1} and
\ref{prop:stab-prodprojective-logtgt3}, we get the first point.

Let $r=1$. We have $-(K_X + D_0) \sim_\lin D_0 + D_3$ and $\nu = 1$. As the polynomial
$\rP_1$ defined in Theorem \ref{theo:logtgt-toricpic-Dv0} is $\rP_1 = 2-x$ and
$0 < \nu < 2$, we deduce that $T_X( - \log D_0)$ is stable with respect to $-(K_X + D_0)$.
\\
The polynomial $\rP_0$ of Theorem \ref{theo:logtgt-toricpic-Dvr} is given by
$\rP_0 = 1-x$. As $-(K_X + D_2) \sim_\lin 2 D_3 + D_0$ and $\nu = 2$, we deduce that
$T_X(- \log D_2)$ is unstable with respect to $-(K_X + D_2)$. 

If $r \geq 2$, then according to Theorem \ref{theo:logtgt-toricpic-Dvr},
$T_X(- \log D_2)$ is unstable with respect to $-(K_X + D)$.
Finally, if $r \geq 1$ and $D \in \{D_2 + D_1, D_2 + D_3\}$, then $T_{X}(- \log D)$ is
unstable with respect to $-(K_X + D)$ (cf. Corollary \ref{cor:logtgt-toricpic-Dvr}).
\end{proof}

\begin{rem}
If $X$ is a smooth toric variety and $D$ an invariant divisor on $X$ such that
$-(K_X + D)$ is ample, by \cite[Theorem 1.2]{BB13}, $(X, D)$ admits a toric log
K\"ahler--Einstein metric if and only if $0$ is the barycenter of the polytope
$P_{(X,D)}$ corresponding to $-(K_X + D)$.
In this case, according to \cite[Theorem 1.4]{Li20}, the orbifold tangent sheaf
$T_{X}(- \log D)$ is polystable with respect to $-(K_X + D)$.
In this paper we studied the stability of $T_{X}(- \log D)$ when $0$ is not the
barycenter of $P_{(X,D)}$. Therefore, we do not have the existence of K\"ahler--Einstein
metrics on these logarithmic pairs $(X, D)$.
\end{rem}

\bibliography{Biblio}
\bibliographystyle{amsplain}

\end{document}